\documentclass[12pt,oneside, a4paper,reqno]{amsart}
\usepackage[utf8]{inputenc}
\usepackage[T1]{fontenc}
\usepackage{lmodern}
\usepackage{etoolbox}
\usepackage[protrusion=false]{microtype}

\usepackage{marginnote}
\usepackage{graphicx}
\usepackage{geometry}
\usepackage[all]{xy}
\pdfpagewidth=\paperwidth
\pdfpageheight=\paperheight
\usepackage[dvipsnames]{xcolor}

\usepackage{mathtools} 
\usepackage{amssymb} 
\usepackage{amsthm} 

\numberwithin{equation}{section}

\usepackage{tikz}
\usepackage{manfnt}
\usetikzlibrary{matrix}
\usetikzlibrary{positioning}
\usetikzlibrary{calc}
\usetikzlibrary{shapes.misc}

\usetikzlibrary{decorations.pathmorphing}
\usepackage{tikz-cd}

\usepackage{enumitem}
\setenumerate{listparindent=\parindent}
\setlist[enumerate]{
	font=\normalfont,
	label=(\roman*),
	topsep=3pt,
	itemsep=-0.3ex,
	partopsep=1ex,
	parsep=1ex}

\def\namedlabel#1#2{\begingroup
	#2%
	\def\@currentlabel{#2}%
	\phantomsection\label{#1}\endgroup
}
\makeatother

\makeatletter
\g@addto@macro\bfseries{\boldmath}
\makeatother

\usepackage{bbm}

\usepackage{datetime}
\usepackage{url}
\usepackage{tensor}

\usepackage[unicode=true,
colorlinks=true,
linkcolor=blue!75!black,
citecolor=blue!75!black,
urlcolor=purple,
breaklinks=true]{hyperref}

\usepackage[
noabbrev,
capitalise,
nameinlink
]{cleveref}

\newtheorem{thm}{Theorem}[section]
\newtheorem{cor}[thm]{Corollary}
\newtheorem{lem}[thm]{Lemma}
\newtheorem{prop}[thm]{Proposition}

\newtheorem*{thm*}{Theorem}

\theoremstyle{definition}
\newtheorem{dfn}[thm]{Definition}

\newtheorem{ntn}[thm]{Notation}

\theoremstyle{remark}
\newtheorem{rmk}[thm]{Remark}

\newtheorem{example}[thm]{Example}

\DeclareMathOperator{\id}{id}

\DeclareMathOperator{\lsp}{span}
\DeclareMathOperator{\clsp}{\overline{\lsp}}

\DeclareMathOperator{\MCE}{MCE}
\DeclareMathOperator{\Aut}{Aut}

\newcommand\restr[2]{{
		\left.\kern-\nulldelimiterspace 
		#1 
		\vphantom{|} 
		\right|_{#2} 
}}

\newcommand{\ol}[1]{\overline{#1}}
\newcommand{\fibre}[2]{\tensor*[^{}_{#1}]{*}{^{}_{#2}}}

\newcommand{\la}{\triangleright}
\newcommand{\ra}{\triangleleft}

\newcommand{\FF}{\mathbb{F}}

\newcommand{\II}{\mathbb{I}}

\newcommand{\NN}{\mathbb{N}}

\newcommand{\TT}{\mathbb{T}}

\newcommand{\ZZ}{\mathbb{Z}}

\newcommand{\Cc}{\mathcal{C}}
\newcommand{\Dd}{\mathcal{D}}
\newcommand{\Ee}{\mathcal{E}}

\newcommand{\Gg}{\mathcal{G}}

\newcommand{\Kk}{\mathcal{K}}
\newcommand{\Ll}{\mathcal{L}}
\newcommand{\Mm}{\mathcal{M}}

\newcommand{\Oo}{\mathcal{O}}
\newcommand{\Pp}{\mathcal{P}}

\newcommand{\Tt}{\mathcal{T}}
\newcommand{\Uu}{\mathcal{U}}

\newcommand{\Xx}{\mathcal{X}}

\newcommand{\Zz}{\mathcal{Z}}

\newcommand{\sCc}{\mkern1mu\Cc}

\newcommand{\bone}{\mathbbm{1}}

\newcommand{\CImap}{\zeta}
\newcommand{\hatCImap}{\zeta_{\star}} 
\newcommand{\frakCImap}{\mathfrak{z}}
\newcommand{\frakhatCImap}{\mathfrak{z}_{\star}}

\begin{document}
	
\DeclareRobustCommand{\subtitle}[1]{\\[1.5em] #1}	

\title[Cocycle homotopy invariance of $K$-theory for self-similar actions]{Self-similar groupoid actions on $k$-graphs, and invariance of $K$-theory for cocycle homotopies}


\author{Alexander Mundey}


\thanks{This research was supported by Australian Research Council grant DP220101631. The first author was supported by University of Wollongong AEGiS CONNECT grant 141765}
\author{Aidan Sims}

%

\date{\today}

\begin{abstract}
We establish conditions under which an inclusion of finitely aligned
left-cancellative small categories induces inclusions of twisted
$C^*$-algebras. We also present an example of an inclusion of finitely
aligned left-cancellative monoids that does not induce a homomorphism even
between (untwisted) Toeplitz algebras. We prove that the twisted $C^*$-algebras of a jointly faithful self-similar action of a countable discrete amenable groupoid on a row-finite $k$-graph with no sources, with respect to homotopic cocycles, have isomorphic $K$-theory.
\end{abstract}

\maketitle

\renewcommand{\subtitle}[1]{}


\section{Introduction}

In this article, we establish that for self-similar actions of amenable groupoids on $k$-graphs, the $K$-theory of the twisted $C^*$-algebras is invariant under $2$-cocycle homotopy.

It is a recurring theme that the $K$-theory of twisted $C^*$-algebras is invariant under homotopies of $2$-cocycles (see for example \cite{Ell84,ELPW10,Gil15,KPS15}).
An early result of this nature is due to Elliott \cite{Ell84}, who showed that the $K$-theory of each rank-$n$ noncommutative torus---which can be regarded as a twisted group $C^*$-algebra of $\ZZ^n$---is isomorphic to the $K$-theory of the (untwisted) group $C^*$-algebra $C^*(\ZZ^n) \cong C(\TT^n)$. Elliott's proof involves inducting on the dimension $n$, using a five-lemma argument based on naturality of the Pimsner--Voiculescu exact sequence. We employ a technique based on Elliott's argument to prove our main theorem.

In \cite{MS?} we introduced twisted $C^*$-algebras for self-similar groupoid actions on $k$-graphs. A $2$-cocycle $\sigma$ for a self-similar action of a groupoid $\Gg$ on a $k$-graph $\Lambda$ is, by definition, a categorical $2$-cocycle on the Zappa--Sz\'ep product category $\Lambda \bowtie \Gg$. The corresponding twisted $C^*$-algebra $C^*(\Lambda \bowtie \Gg, \sigma)$ agrees, in the untwisted case, with the algebras of~ \cite{NekCP,LRRW18,ABRW}. It is natural to ask about invariance of $K$-theory as $\sigma$ varies continuously. Our main theorem can be summarised as follows.

\begin{thm*}[{\cref{thm:cocycle_homotopy_independence}}]
	Fix $k \ge 0$. Suppose that $\Gg$ is a countable discrete amenable groupoid acting self-similarly and jointly faithfully on a row-finite $k$-graph $\Lambda$ with no sources. If $\sigma_0$ and $\sigma_1$ are homotopic $2$-cocycles on $\Lambda \bowtie \Gg$, then
	\[
	K_*(C^*(\Lambda \bowtie \Gg, \sigma_0)) \cong K_*(C^*(\Lambda \bowtie \Gg, \sigma_1)).
	\]
\end{thm*}

Indeed, like Elliott, we prove the stronger statement that the twisted $C^*$-algebras for the cocycles along the homotopy assemble into a $C([0,1])$-algebra for which the point-evaluation $C^*$-homomorphisms all induce isomorphisms in $K$-theory. As in Elliott's argument, we induct on $k$. However, we use the Pimsner exact sequence for Cuntz--Pimsner algebras \cite{Pim97} rather than the Pimsner--Voiculescu sequence for crossed products \cite{PimVoi}. In particular, we show that the twisted $C^*$-algebra $C^*(\Lambda \bowtie \Gg,\sigma)$, with $\Lambda$ a $(k+1)$-graph, can be realised as the Cuntz--Pimsner algebra of a $C^*$-correspondence over $C^*(\Gamma \bowtie \Gg,\sigma)$, where $\Gamma \subseteq \Lambda$ is a sub-$k$-graph.

To do so, we investigate how inclusions of finitely aligned left-cancellative small categories correspond to inclusions of twisted versions of the associated $C^*$-algebras introduced by Spielberg~\cite{Spi20}. This turns out to be more complicated than we expected. We give an example (\cref{example:concordance_counterexample}) of an inclusion of finitely aligned left-cancellative monoids that does not induce a $*$-homomorphism between the associated Toeplitz algebras.

The paper is structured as follows.  \cref{sec:prelims} contains preliminary material on self-similar actions and Zappa--Sz\'ep products. In \cref{sec:twisted_category_algs} we introduce twisted $C^*$-algebras of finitely aligned left-cancellative small categories, generalising the untwisted algebras of Spielberg \cite{Spi20}. In particular, we identify sufficient conditions for an inclusion of categories to induce inclusions of $C^*$-algebras. In \cref{sec:cocycle_homotopy} we introduce the $C^*$-algebra of a finitely aligned left-cancellative small category twisted by a homotopy of $2$-cocycles, and use it to prove our main result, \cref{thm:cocycle_homotopy_independence}.

\section{Preliminaries}
\label{sec:prelims}

\subsection{Categories, groupoids, and \texorpdfstring{\boldmath$k$}{k}-graphs}
Throughout this article, $\Cc$ denotes a countable discrete small category. We identify $\Cc$ with its set of morphisms and denote its set of objects (identified with the corresponding identity morphisms) by $\Cc^0$. The domain and codomain maps become maps $r \colon \Cc \to \Cc^0$ and $s \colon \Cc \to \Cc^0$ taking a morphism to its \emph{range} and \emph{source}. We write $\Cc^n$ for the collection of composable $n$-tuples in $\Cc$. That is $(c_1,\ldots,c_n) \in \Cc^n$ if $s(c_i) = r(c_{i+1})$ for all $i$. We extend $r$ and $s$ to $\Cc^n$ by $r(c_1,\ldots,c_n) = r(c_1)$ and $s(c_1,\ldots,c_n) = s(c_n)$. For $c_1,c_2 \in \Cc$ we define
\begin{gather*}
	c_1 \Cc \coloneqq \{ c_1 c \colon (c_1,c) \in \Cc^2\},
	\quad \Cc c_2 \coloneqq \{ c c_2 \colon (c,c_2) \in \Cc^2\}, \quad \text{and}\\
	c_1 \Cc c_2 \coloneqq \{c_1 c c_2 : (c_1, c, c_2) \in \Cc^3\};
\end{gather*}
if $c_1, c_2 \in C^0$, then $c_1 \Cc c_2 = c_1 \Cc \cap \Cc c_2$.

A \emph{groupoid} $\Gg$ is a small category such that every $g \in \Gg$ has an inverse $g^{-1} \in \Gg$ such that $gg^{-1} = r(g)$ and $g^{-1}g = s(g)$.

Let $k \in \NN$. A \emph{$k$-graph} is a countable small category $\Lambda$ together with a functor $d \colon \Lambda \to \NN^k$, called the \emph{degree functor}, such that composition in $\Lambda$ satisfies the \emph{unique factorisation property}: for each $\lambda \in \Lambda$ and $m,n \in \NN^k$ such that $d(\lambda) = m + n$ there exist unique $\mu,\nu \in \Lambda$ such that $\lambda = \mu \nu$, $d(\mu) = m$, and $d(\nu) = n$. If $d(\lambda) = n$ we say that $\lambda$ has \emph{degree} $n$. We write $\Lambda^n \coloneqq d^{-1}(n)$. We call elements of $\Lambda^0$ \emph{vertices}.
We denote the element of $\NN^k$ with a $1$ in the $i$-th component and $0$ elsewhere by $e_i$.
For $m,n \in \NN^k$ we write $m \le n$ if $m_i \le n_i$ for all $i$.

A $1$-graph is just the category of finite paths $E^*$ in a directed graph $E$ with composition given by concatenation of paths and $d \colon E^* \to \NN$ taking a path to its length. 
As in \cite{KPS11} we allow $0$-graphs: by $\NN^0$ we mean the $1$-element monoid $\{0\}$, so a $0$-graph is just a set $\Lambda$ regarded as a small category consisting solely of identity morphisms and the degree map is the unique map $d \colon \Lambda \to \{0\}$.

A $k$-graph $\Lambda$ is \emph{row-finite} if $|v\Lambda^n| < \infty$ for all $n \in \NN^k$ and $v \in \Lambda^0$. It has \emph{no sources} if $v\Lambda^n \ne \varnothing$ for all $n \in \NN^k$ and $v \in \Lambda^0$.
It is \emph{locally convex} if for every $1 \le i,j \le k$ with $i \ne j$, whenever $e \in \Lambda^{e_i}$ and $r(e) \Lambda^{e_j} \ne \varnothing$, we have $s(e) \Lambda^{e_j} \ne \varnothing$. Every $k$-graph with no sources is locally convex.

For each $n \in \NN^k$ we define, as in \cite{RSY03},
\[
\Lambda^{\le n} \coloneqq
\{
\lambda \in \Lambda \mid d(\lambda) \le n \text{ and if } d(\lambda)_i < n_i \text{ then } s(\lambda)\Lambda^{e_i} = \varnothing
\}.
\]
It is potentially confusing that $\Lambda^{\le n} \ne \bigcup_{m \le n} \Lambda^m$; however the notation is, by now, standard.

If $d(\lambda)  = n$, then the condition for membership of $\lambda$ in $\Lambda^{\le n}$ is vacuous. So $\Lambda^n \subseteq \Lambda^{\le n}$. An induction using \cite[Lemma~3.12]{RSY03} (used implicitly in the proof of \cite[Proposition~3.11]{RSY03}) shows that if $\Lambda$ is a locally convex $k$-graph then
\begin{equation}\label{eq:le_factorisation}
	\Lambda^{\le m+n} = \Lambda^{\le m} \Lambda^{\le n}
\end{equation}
for all $m,n \in \NN^k$.

The following lemma is used implicitly in \cite{RSY03} but never stated explicitly. We need it for the proof of \cref{lem:k_graph_exhaustive}.
\begin{lem}\label[lem]{lem:le_extensions_equal}
	Let $\Lambda$ be a locally convex $k$-graph. Suppose that $\mu,\nu \in \Lambda^{\le n}$ and $\mu \Lambda \cap \nu \Lambda \ne \varnothing$. Then $\mu = \nu$. More generally, if $\nu \in \Lambda^{\le n}$, $d(\mu) \le n$, and $\mu \Lambda \cap \nu \Lambda \ne \varnothing$, then $\nu \in \mu \Lambda$.
\end{lem}
\begin{proof}
	By \labelcref{eq:le_factorisation} we have $\mu \in \nu \Lambda^{\le n - d(\nu)}$. Since $\nu \in \Lambda^{\le n}$, we have $s(\nu) \Lambda^{e_i} \ne \varnothing$ whenever $e_i \le n - d(\nu)$. Hence, $s(\nu)\Lambda^{\le n - d(\nu)} = \{s(\nu)\}$. So $\mu = \nu s(\nu) = \nu$.
	
	For the second statement suppose that $\nu \in \Lambda^{\le n}$, $d(\mu) \le n$, and $\mu \Lambda \cap \nu \Lambda \ne \varnothing$. Since $d(\mu) \le n$, \cref{eq:le_factorisation} gives $\nu' \in \Lambda^{\le d(\mu)}$ and $\nu'' \in \Lambda^{\le n - d(\mu)}$ such that $\nu = \nu' \nu''$. Since $\mu \Lambda \cap \nu \Lambda \ne \varnothing$ we have $\mu \Lambda \cap \nu' \Lambda \ne \varnothing$ with $\mu,\nu' \in \Lambda^{\le d(\mu)}$. So the first statement gives $\mu = \nu'$.
\end{proof}

\subsection{Zappa--Sz\'ep products and self-similar actions}

For details of the following, see \cite[Section~3]{MS?}. A \emph{left action} of a small category $\Cc$ on a small category $\Dd$ with $\Cc^0 = \Dd^0$ consists of a map
\[
\la \colon \Cc \fibre{s}{r} \Dd  \coloneqq \{(c,d) \in \Cc \times \Dd \mid s(c) = r(d)\} \to \Dd
\] such that  $r(d) \la d = d$, $c \la s(c) = r(c)$, $r(c \la d) = r(c)$, and $(cc') \la d = c \la (c' \la d)$ for all $c,c' \in \Cc$ and $d \in \Dd$ for which the formulas make sense.
A \emph{right action} $\ra \colon \Cc \fibre{s}{r} \Dd \to \Cc$ is defined symmetrically.

Following \cite[Definition~3.1]{MS?}, a \emph{matched pair} $(\Cc,\Dd)$ consists of a pair of small categories $\Cc$ and $\Dd$ with $\Cc^0 = \Dd^0$, together with a left action $\la \colon \Cc \fibre{s}{r} \Dd \to \Dd$ of $\Cc$ on $\Dd$ and a right action $\ra \colon \Cc \fibre{s}{r} \Dd \to \Cc$ of $\Dd$ on $\Cc$ such that $s(c \la d) = r(c \ra d)$ for all $(c,d) \in \Cc \fibre{s}{r} \Dd$ and such that for all $(c_1,c_2,d_1,d_2) \in \Cc^2 \fibre{s}{r} \Dd^2$ we have
\[
c_2 \la (d_1d_2) = (c_2 \la d_1)((c_2 \ra d_1) \la d_2 ) \quad \text{ and } \quad
(c_1c_2) \ra d_1 = (c_1 \ra (c_2 \la d_1)) (c_2 \ra d_1).
\]
\begin{dfn}[cf.~{\cite[Definition~3.6]{MS?}}]
	Let $(\Cc,\Dd)$ be a matched pair of small categories. The \emph{Zappa--Sz\'ep} product $\Dd \bowtie \Cc$ is the small category with objects $\Cc^0 = \Dd^0$ and morphisms $\{dc \mid (d,c) \in \Dd \fibre{s}{r} \Cc \}$, in which the range and source maps are given by $r(dc) = r(d)$ and $s(dc) = s(c)$, and composition is defined by the formula
	\[
	d_1c_1 d_2 c_2 = d_1(c_1 \la d_2) (c_1 \ra d_2) c_2
	\]
	whenever $d_1c_1, d_2c_2 \in \Dd \bowtie \Cc$ with $s(c_1) = r(d_2)$.
\end{dfn}

\begin{rmk}
	Our notation for the Zappa--Sz\'ep product of the matched pair $(\Cc,\Dd)$ is reversed relative to \cite{MS?}: there it would have been denoted $\Cc \bowtie \Dd$ rather than $\Dd \bowtie \Cc$. 
	The anonymous referee has pointed out that the convention we are now using is the one that is consistent with the literature \cite{BKQS18,Bri05,BPRRW17,LV22} and reflects the fact that as a set the Zappa--Sz\'ep product is $\Dd * \Cc$.
\end{rmk}

By \cite[Lemma~3.5]{MS?} the Zappa--Sz\'ep product $\Dd \bowtie \Cc$ is indeed a small category. By \cite[Proposition~3.13]{MS?} it is characterised by a unique-factorisation property: if $\Ee$ is a small category and $\Cc$ and $\Dd$ are wide subcategories of $\Ee$ such that for any $e \in \Ee$ there are unique $d \in \Dd$ and $c \in \Cc$ such that $e = dc$, then $\Ee$ is isomorphic to a Zappa--Sz\'ep product $\Dd \bowtie \Cc$.


In \cite[Proposition~{3.29}]{MS?} it is shown that $k$-graphs can be described as Zappa--Sz\'ep products. We recap this briefly in the following example.
\begin{example} \label[example]{ex:k-graph_zs}
	
	Take $k \ge 0$ and let $\Lambda$ be a $k$-graph. Fix $p,q \ge 0$ such that $p+q = k$. Write $\NN^p \times \{0\}$ for the submonoid of $\NN^k$ consisting of tuples whose last $q$ coordinates are 0, and write $\Lambda^{\NN^p \times \{0\}} \coloneqq d^{-1}(\NN^p \times \{0\}) \subseteq \Lambda$. Identifying $\NN^p \times \{0\}$ with $\NN^p$ in the obvious way, $\Lambda^{\NN^p \times \{0\}}$ is a $p$-graph with respect to the restriction of the degree functor on $\Lambda$. Similarly, we can identify $\{0\} \times \NN^q \subseteq \NN^k$ with $\NN^q$ and then $\Lambda^{\{0\} \times \NN^q} \coloneqq d^{-1}(\{0\} \times \NN^q)$ is a $q$-graph.
	
	If $\lambda \in \Lambda^{\NN^p \times \{0\}}$ and $\mu \in s(\lambda)\Lambda^{\{0\} \times \NN^q}$, then uniqueness of factorisation in $\Lambda$ gives $\mu' \in \Lambda^{\{0\} \times \NN^q}$ and $\lambda' \in \Lambda^{\NN^p \times \{0\}}$ such that $\lambda \mu = \mu' \lambda'$. Setting $\lambda \la \mu = \mu'$ defines a left action of $\Lambda^{\NN^p \times \{0\}}$ on $\Lambda^{\{0\} \times \NN^q}$ and setting $\lambda \ra \mu = \lambda'$ defines a right action of $\Lambda^{\{0\} \times \NN^q}$ on $\Lambda^{\NN^p \times \{0\}}$.	
	Uniqueness of factorisation in $\Lambda$ implies that $\Lambda \cong \Lambda^{\{0\} \times \NN^q} \bowtie \Lambda^{\NN^p \times \{0\}} $.
	In particular, if $\Lambda$ is a $(k+1)$-graph, then $\Gamma \coloneqq \Lambda^{\NN^k \times \{0\}}$ is a $k$-graph and $\Lambda^{\{0\} \times \NN}$ is the path category $E^*$ of a directed graph $E$ such that $\Lambda \cong E^* \bowtie \Gamma$.
\end{example}

Zappa--Sz\'ep products also capture self-similar actions of groupoids on $k$-graphs (see~\cite[Proposition~3.32]{MS?}). We use the following definition of a self-similar action.

\begin{dfn}[{\cite[Definition~3.33]{MS?}}]
	\label{dfn:self-similar_action}
	Let $\Lambda$ be a $k$-graph and let $\Gg$ be a discrete groupoid with $\Gg^0 = \Lambda^0$. A \emph{self-similar action} of $\Gg$ on $\Lambda$ is a matched pair $(\Gg,\Lambda)$ such that $d(g \la \lambda) = d(\lambda)$ for all $g,\lambda$. We say that \emph{$\Gg$ is a discrete groupoid acting self-similarly on $\Lambda$.}
\end{dfn}

\section{Twisted \texorpdfstring{\boldmath$C^*$}{C*}-algebras of finitely aligned left-cancellative small categories}
\label{sec:twisted_category_algs}

In this section we introduce twisted $C^*$-algebras for finitely aligned left-cancellative small categories $\Cc$ as introduced in \cite{Spi20}. We use the notation and setup of \cite{BKQS18}. The examples we have in mind are $\Cc = \Gg$ a discrete group(oid), or $\Cc = \Lambda$ a row-finite $k$-graph with no sources (see Examples \ref{ex:groupoid}~and~\ref{ex:k-graph}).

\subsection{Finitely aligned left-cancellative small categories}
A small category $\Cc$ is \emph{left-cancellative} if whenever $a,b,c \in \Cc$ satisfy $s(a) = r(b) = r(c)$,
\[
ac = ab \quad \text{implies} \quad c = b.
\]
Given a left-cancellative small category $\Cc$, we define an equivalence relation on $\Cc$ by
\[
a \sim b \qquad \text{ if and only if} \qquad \text{ there is an invertible } c \in \Cc \text{ such that } a = bc.
\]
Equivalently, $a \sim b$ if and only if $a \sCc = b \Cc$. That is, $a$ and $b$ generate the same principal right ideals. 

We extend the notion of equivalence to subsets of $\Cc$. If $A,B \subseteq \Cc$ then we say that $A \sim B$ if for every $a \in A$ there exists $b \in B$ such that $a \sim b$ and for every $b \in B$ there exists $a \in A$ such that $b \sim a$. If $A \sim B$, then
\begin{equation}\label{eq:F_unions_equal}
	\bigcup_{a \in A} a\sCc = \bigcup_{b \in B} b \sCc,
\end{equation}
but the converse does not typically hold.

If $a,a' \in A$ satisfy $a \in a' \Cc$ then $a \sCc \subseteq a' \Cc$. We say that $A \subseteq \Cc$ is \emph{independent} if for all distinct $a,a' \in A$ we have $a \notin a' \Cc$.

\begin{lem}[{cf. \cite[p.~1350]{BKQS18}}]
	\label[lem]{lem:independent_unions_equal} Let $\Cc$ be a left-cancellative small category.
	If $A,B \subseteq \Cc$ are independent and \eqref{eq:F_unions_equal} holds, then $A \sim B$, and $|A| = |B|$.
\end{lem}
\begin{proof}
	Fix $a \in A$. Since \eqref{eq:F_unions_equal} holds, there exists $b \in B$ such that $a \in b \sCc$ and there exists an $a' \in A$ such that $b \in a' \Cc$. Hence, $a \sCc \subseteq b \sCc \subseteq a' \Cc$. Independence of $A$ implies that $a = a'$ so $a \sCc = b \sCc$ and therefore $a \sim b$. If $b' \in B$ also satisfies $a \sim b'$ then $b \sim b'$ and independence of $B$ implies that $b = b'$. So there is a unique function $a \mapsto b_a$ from $A$ to $B$ such that $a \sim b_a$ for all $a \in A$.
	A symmetric argument yields a function $b \mapsto a_b$ from $B \to A$
	such that $b \sim a_b$ for all $b \in B$, and this function is inverse to $a \mapsto b_a$.
\end{proof}

We say that a left-cancellative small category $\Cc$ is \emph{finitely aligned} if for all $a,b \in \Cc$ there is a finite subset $F \subseteq \Cc$ such that
\begin{equation}\label{eq:finitely_aligned}
	a \sCc \cap b \sCc = \bigcup_{c \in F} c \sCc \eqqcolon F \Cc.
\end{equation}
Since $F$ is finite, by passing to a subset we can assume that $F$ is independent, and therefore unique up to equivalence by \cref{lem:independent_unions_equal}. In this paper, we work exclusively with finitely aligned left-cancellative small categories.

If $\Cc$ is finitely aligned, an induction on $|A|$ shows that for any finite set $A \subseteq \Cc$ there is a finite independent set $F \subseteq \Cc$, unique up to equivalence, such that
\[
\bigcap_{a \in A} a \sCc = \bigcup_{c \in F} c \sCc.
\]
Following \cite{BKQS18} we write $\bigvee A$ for a choice of such a finite independent set $F$. If $a,b \in \Cc$ then we write $a \vee b$ instead of $\bigvee \{a,b\}$.

Let $v \in \Cc^0$. 	A subset $A \subseteq v\Cc$ is \emph{exhaustive} if for every $c \in v\Cc$ there exists $a \in A$ such that $c \sCc \cap a\sCc \ne \varnothing$.

\begin{lem}[{\cite[Lemma~2.3]{BKQS18}}]
	Let $\Cc$ be a left-cancellative small category. For $v \in \Cc^0$, if $A \subseteq v \Cc$ is exhaustive and $A \sim B$, then $B \subseteq v \Cc$ is exhaustive.
\end{lem}

\begin{example}\label[example]{ex:groupoid}
	Every groupoid (and hence every group) $\Gg$  is a finitely aligned left-cancellative small category since $g \Gg = r(g)\Gg$ for all $g \in \Gg$.
\end{example}

\begin{example}\label[example]{ex:k-graph}
	The path category $E^*$ of a directed graph $E$ is a  finitely aligned left-cancellative small category.
	All $k$-graphs are left cancellative: if $\Lambda$ is a $k$-graph and $\mu,\nu, \lambda \in \Lambda$ satisfy $\mu\lambda = \nu\lambda =: \gamma$, then in particular
	$d(\mu) + d(\lambda) = d(\gamma) = d(\nu) + d(\lambda)$, so the uniqueness condition in the factorisation
	property for $\gamma$ implies that $\mu$ and $\nu$ are equal. For $k \ge 2$ some $k$-graphs are
	\emph{not} finitely aligned (\cite[Examples 3.1~and~5.2]{RS05}). However, all row-finite $k$-graphs are finitely aligned: if $\Lambda$ is a row-finite $k$-graph and $\mu,\nu \in \Lambda$, then the collection of \emph{minimal common extensions} of $\mu$ and $\nu$,
	\begin{equation}\label{eq:MCE}
		\MCE(\mu,\nu) \coloneqq \mu \Lambda \cap \nu \Lambda \cap \Lambda^{d(\mu) \vee d(\nu)}
	\end{equation}
	is finite and $\mu \Lambda \cap \nu \Lambda = \MCE(\mu,\nu) \Lambda$ by the factorisation property.
\end{example}

\begin{lem}\label[lem]{lem:sub_kgraph_fin_align}
	Fix $S \subseteq \{1,\ldots,k\}$ and let $\NN^{S} \coloneqq \{ n \in \NN^k \mid n_i = 0 \text{ for } i \not\in S\}$. Let $\Lambda$ be a finitely aligned $k$-graph. Then $\Lambda^{\NN^S}$ is finitely aligned.
\end{lem}
\begin{proof}
	Fix $\alpha,\beta \in \Lambda^{\NN^S}$. Then $\alpha \vee \beta $, as calculated in $\Lambda$, is a subset of $\Lambda^{\NN^S}$ and hence is equal to $\alpha \vee \beta$ as calculated in $\Lambda^{\NN^S}$.
\end{proof}

\begin{lem}\label[lem]{lem:groupoid_gives_finitely_aligned}
	Let $\Gg$ be a groupoid and let $\Cc$ be a finitely aligned left-cancellative small category such that $(\Gg,\Cc)$ is a matched pair. Then $\Cc \bowtie \Gg$ is left cancellative and finitely aligned.
\end{lem}
\begin{proof}
	The left-cancellativity of $\Cc \bowtie \Gg$ follows from {\cite[Lemma~3.15]{MS?}} since every element of $\Gg$ is invertible.
	
	For finite alignment let $cg \in \Cc \bowtie \Gg$. Since $g$ is invertible we have $cg(\Cc \bowtie \Gg) = c(\Cc \bowtie \Gg)$. Fix $c_1,c_2 \in \Cc$ and suppose that $c_1 (\Cc \bowtie \Gg) \cap c_2 (\Cc \bowtie \Gg) \ne \varnothing$. Fix $c_1c_1'g_1 = c_2 c_2'g_2 \in \Cc \bowtie \Gg$.
	Uniqueness of factorisation gives $c_1c_1' = c_2 c_2'$, so $c_1 \Cc \cap c_2 \Cc \ne \varnothing$. Since $\Cc$ is finitely aligned, there are a finite subset $F \subseteq \Cc$ and $c \in F$ such that $c_1c_1' = c_2 c_2' = c c'$ for some $c' \in \Cc$. Hence, $cc'g_1 = cc'g_2$ and it follows that $c_1 (\Cc \bowtie \Gg) \cap c_2 (\Cc \bowtie \Gg) = \bigcup_{c \in F} c (\Cc \bowtie \Gg)$.
\end{proof}

\subsection{Twisted \texorpdfstring{\boldmath$C^*$}{C*}-algebras}
To introduce twisted $C^*$-algebras of small categories we need to recall the notion of a $\TT$-valued $2$-cocycle.
\begin{dfn}\label[dfn]{dfn:2-cocycle}
	A \emph{$\TT$-valued $2$-cocycle} on a category $\Cc$ is a map $\sigma \colon \Cc^2 \to \TT$ such that for every $(c_1,c_2,c_3) \in \Cc^3$,
	\[
	\sigma(c_2,c_3)\sigma(c_1,c_2c_3) = \sigma(c_1,c_2)\sigma(c_1c_2,c_3),
	\]
	and such that $\sigma(r(c),c) = 1 = \sigma(c,s(c))$ for all $c \in \Cc$.
\end{dfn}
\begin{rmk}\label[rmk]{rmk:coycles normalised}
	The condition $\sigma(r(c),c) = 1 = \sigma(c,s(c))$ for all $c \in \Cc$ is often emphasised by saying that $\sigma$ is \emph{normalised}, but in this paper all cocycles are normalised, so we drop the adjective.
\end{rmk}

Given a finite family $\Pp$ of pairwise commuting projections in a $C^*$-algebra $A$, we write $\bigvee \Pp$ for the smallest projection in $A$ that dominates every element of $\Pp$. Explicitly,

\begin{equation}\label{eq:sum_product_formula}
	\bigvee \Pp = \sum_{\varnothing \ne F \subseteq \Pp} (-1)^{|F|-1} \prod_{P\in F} P.
\end{equation}
We take the convention that $\bigvee \varnothing = 0$.

\begin{dfn}[cf.~{\cite[Definition~3.1]{BKQS18}}]
	\label[dfn]{dfn:twisted_lcsc_algebra}
	Let $\Cc$ be a finitely aligned left-cancellative small category and let $\sigma$ be a $\TT$-valued $2$-cocycle on $\Cc$. A \emph{$\sigma$-twisted representation} of $\Cc$ in a $C^*$-algebra $A$ is a map $S \colon \Cc \to A$, $c \mapsto S_c$ such that each $S_c$ is a partial isometry and
	\begin{enumerate}[labelindent=0pt,labelwidth=\widthof{\ref{itm:cf-mult}},label=(R\arabic*), ref=(R\arabic*),leftmargin=!]
		\item \label{itm:cf-mult} $S_{c_1}S_{c_2} = \delta_{s(c_1), r(c_2)} \sigma(c_1,c_2) S_{c_1c_2}$ for all $c_1,c_2 \in \Cc$,
		\item \label{itm:cf-source} $S_c^* S_{c} = S_{s(c)}$ for all $c \in \Cc$, and
		\item \label{itm:cf-range} $\displaystyle S_{c_1}S_{c_1}^* S_{c_2}S_{c_2}^* = \bigvee_{c \in F} S_{c}S_{c}^*$ for all $c_1,c_2 \in \Cc$ and any finite independent set $F$ satisfying $c_1 \Cc \cap c_2 \Cc = F \Cc$, with the convention that if $c_1 \Cc \cap c_2 \Cc = \varnothing$, then $F = \varnothing$ is such a finite independent set, so $S_{c_1}S_{c_1}^* S_{c_2}S_{c_2}^* = 0$.
	\end{enumerate}
	We say that $S$ is \emph{covariant} if, in addition, for all $v \in \Cc^0$,
	\begin{enumerate}[labelindent=0pt,labelwidth=\widthof{\ref{itm:cf-mult}},label=(R\arabic*), ref=(R\arabic*),leftmargin=!]
		\setcounter{enumi}{3}
		\item \label{itm:cf-ck}
		$\displaystyle S_v = \bigvee_{c \in F}  S_{c}S_{c}^*$ for all finite exhaustive $F \subseteq v\Cc$.
	\end{enumerate}
\end{dfn}

\begin{rmk}\label[rmk]{rmk:independence_not_needed}
	Suppose that $\Cc$ is a finitely aligned left-cancellative small category and $\sigma$ is a $\TT$-valued $2$-cocycle on $\Cc$. Let $c_1,c_2 \in \Cc$ and let $F$ be a finite set generating $c_1 \Cc \cap c_2 \Cc$. As discussed immediately following \labelcref{eq:finitely_aligned} there is an independent subset $G$ of $F$ that generates $c_1 \Cc \cap c_2 \Cc$. If $c \in F$ then there exists $c' \in G$ such that $c \in c'\Cc$ and hence $S_cS_c^* \le S_{c'}S_{c'}^*$ by \labelcref{itm:cf-mult}. Hence, $\bigvee_{c \in F} S_{c}S_{c}^* = \bigvee_{c \in G} S_{c}S_{c}^* = S_{c_1}S_{c_1}^* S_{c_2}S_{c_2}^*$ by \labelcref{itm:cf-range}.
\end{rmk}

We consider two twisted $C^*$-algebras associated to each $2$-cocycle on a finitely aligned left-cancellative small category. A standard argument following the lines of \cite[Proposition~7.4]{MS?} or an appeal to \cite{Lor10} establishes that such $C^*$-algebras exist.

\begin{dfn}\label[dfn]{dfn:univ_twisted_lcsc_alg}
	Let $\Cc$ be a finitely aligned left-cancellative small category, and let $\sigma$ be a $2$-cocycle on $\Cc$. The \emph{$\sigma$-twisted Toeplitz algebra of $\Cc$} is the $C^*$-algebra $\Tt C^*(\Cc,\sigma)$ generated by a universal $\sigma$-twisted representation $t \colon \Cc \to \Tt C^*(\Cc,\sigma)$; that is, $\Tt C^*(\Cc,\sigma) = C^*(\{t_c \mid c \in \Cc\})$, and for any $\sigma$-twisted representation $S \colon \Cc \to A$ there is a unique $*$-homomorphism $\Phi \colon \Tt C^*(\Cc,\sigma) \to A$ such that $S = \Phi \circ t$.
	
	The \emph{$\sigma$-twisted $C^*$-algebra of $\Cc$} is the $C^*$-algebra $C^*(\Cc,\sigma)$ generated by a universal $\sigma$-twisted covariant representation $s \colon \Cc \to C^*(\Cc,\sigma)$. Let $I$ be the ideal of $\Tt C^*(\Cc,\sigma)$ generated by $\{ S_v - \bigvee_{c \in F}  S_{c}S_{c}^* \mid v \in \Cc^0, F \subseteq v\Cc \text{ is finite exhaustive}\}$. Then $C^*(\Cc,\sigma) \cong \Tt C^*(\Cc,\sigma)/ I$.
\end{dfn}
The ``untwisted'' algebra  $C^*(\Cc,1)$ of a countable finitely aligned left-cancellative small category coincides with the groupoid $C^*$-algebra $C^*(G_2|_{\partial \Lambda})$  of \cite[Theorem~10.15]{Spi20}.

\begin{rmk}[Notational conventions for representations of categories]
	There are numerous different representations of a category $\Cc$ appearing in this paper, so, taking on board the comments of a helpful anonymous referee, we have tried to adopt a consistent convention. We denote generic families of partial isometries indexed by elements of the category $\Cc$ and satisfying \labelcref{itm:cf-source}~and~\labelcref{itm:cf-range} and some analogue of~\labelcref{itm:cf-mult} by $S$ throughout. But for universal families we use lower-case letters, and different letters and fonts as follows: $t$ is for families twisted by a cocycle as in~\labelcref{itm:cf-mult} that do not (necessarily) satisfy~\labelcref{itm:cf-ck}; $s$ is for those that are required to satisfy~\labelcref{itm:cf-ck}. When we introduce universal representations twisted by a continuous family of cocycles in \cref{sec:cocycle_homotopy}, we use fraktur font ($\mathfrak{t}$ and $\mathfrak{s}$) to distinguish them.
\end{rmk}

Let $\Mm(A)$ denote the multiplier algebra of a $C^*$-algebra $A$, and let $\Uu\Zz \Mm(A)$ denote the group of unitaries in the center of the multiplier algebra. 
The following is a generalisation of \cite[Lemma~3.4]{BKQS18}.
\begin{lem}\label[lem]{lem:twisted_category_algebra_identities}
	Let $\Cc$ be a finitely aligned left-cancellative small category and let $A$ be a $C^*$-algebra. Suppose that $S \colon \Cc \to A$ satisfies \labelcref{itm:cf-source} of \cref{dfn:twisted_lcsc_algebra}, and that $\omega \colon \Cc^2 \to \Uu\Zz \Mm(A)$ satisfies $S_{c_1}S_{c_2} = \omega(c_1,c_2)S_{c_1 c_2}$ for all $(c_1,c_2) \in \Cc^2$. Then
	\begin{enumerate}
		\item \label{itm:identities_0} 
		$S_{r(c)} S_c = S_c = S_c S_{s(c)}$ for all $c \in \Cc$;
		\item \label{itm:identities_1} for every invertible $c \in \Cc$ we have $S_{c^{-1}} ={\omega(c,c^{-1})} S_{c}^{*}$ and $S_{c}S_{c}^* = S_{r(c)}$;
		\item \label{itm:identities_2} if $a \sim b$ in $\Cc$ then $S_a S_a ^* = S_b S_b^*$; and
		\item \label{itm:identities_3} if $A, B \subseteq \Cc$ satisfy $A \sim B$, then
		$
		\bigvee_{a \in A} S_aS_a^* = \bigvee_{b \in B} S_b S_b^*.
		$
	\end{enumerate}
\end{lem}
\begin{proof}
	\labelcref{itm:identities_0} First note that for $v \in \Cc^0$, the condition~\labelcref{itm:cf-source} implies first that $S_v$ is self-adjoint, and then that $S_v$ is idempotent, hence a projection. Consequently, for $c \in \Cc$, we have 
	$S_c = \omega(r(c),c)^* S_{r(c)}S_c = \omega(r(c),c)^* S_{r(c)}^2 S_c = S_{r(c)} S_c$. A symmetric argument gives $S_c = S_c S_{s(c)}$.

	\labelcref{itm:identities_1} Suppose that $c \in \Cc$ is invertible. Then
	\begin{align*}
		S_{c^{-1}}
		&= S_{s(c)c^{-1}}
		= S_c^* S_c S_{c^{-1}} \\
		&= \omega(c,c^{-1}) S_c^* S_{r(c)}
		= \omega(c,c^{-1}) (S_{r(c)} S_c)^*
		= \omega(c,c^{-1}) S_c^*.
	\end{align*}
	It follows that
	\[
	S_c S_c^* = {\omega(c,c^{-1})}^* S_c S_{c^{-1}} =  {\omega(c,c^{-1})}^* \omega(c,c^{-1}) S_{r(c)} = S_{r(c)}.
	\]
	\labelcref{itm:identities_2} Suppose that $a \sim b$. Then there exists an invertible $c \in \Cc$ such that $a = bc$. Now,
	\[
	S_a S_a^* = S_{bc} S_{bc}^* = {\omega(b,c)}^* {\omega(b,c)} S_b S_c S_c^* S_b^* = S_b S_{r(c)} S_b^* = S_b S_b^*.
	\]
	\labelcref{itm:identities_3} This follows immediately from \labelcref{itm:identities_2}.
\end{proof}
\begin{example}
	Let $\Gg$ be a discrete groupoid. Then \cref{lem:twisted_category_algebra_identities} implies that if $S \colon \Gg \to A$ satisfies \labelcref{itm:cf-mult} and \labelcref{itm:cf-source}, then $S_g$ is a partial unitary for all $g \in \Gg$, and $S$ is a $\sigma$-twisted unitary representation of $S_g$ as in \cite{Renault80}. In particular, $S$ automatically satisfies \labelcref{itm:cf-range} and \labelcref{itm:cf-ck} and the $\sigma$-twisted $C^*$-algebra of $\Gg$ of \cref{dfn:univ_twisted_lcsc_alg} is the usual $\sigma$-twisted groupoid $C^*$-algebra.
\end{example}

Our main motivating example is the twisted $C^*$-algebras of self-similar actions of groupoids $\Gg$ on row-finite $k$-graphs $\Lambda$ with no sources introduced in \cite{MS?}. So we show that the preceding definition, in the situation where $\Cc = \Lambda \bowtie \Gg$ is the Zappa--Sz\'ep product coming from a self-similar action, is equivalent to the definition of a Cuntz--Krieger $(\Gg,\Lambda; \sigma)$-family in the sense of \cite[Definition~7.1]{MS?}. 

With $\MCE(\mu,\nu)$ as in \labelcref{eq:MCE}, recall from \cite[Definition~7.1]{MS?} that a \emph{Toeplitz--Cuntz--Krieger $(\Gg, \Lambda; \sigma)$-family} in a $C^*$-algebra $A$ is a map $S \colon \Lambda \bowtie \Gg \to A$ that satisfies \labelcref{itm:cf-mult}~and~\labelcref{itm:cf-source} of \cref{dfn:twisted_lcsc_algebra} and such that
\begin{enumerate}[labelindent=0pt,labelwidth=\widthof{\ref{itm:TCK3}},label=(TCK\arabic*), ref=(TCK\arabic*),leftmargin=!]\setcounter{enumi}{2}
	\item\label{itm:TCK3} for all $\mu,\nu \in \Lambda$, we have $S_\mu S^*_\mu S_\nu S^*_\nu = \sum_{\lambda \in \MCE(\mu,\nu)} S_\lambda S^*_\lambda$.
\end{enumerate}
The map $S$ is called a \emph{Cuntz--Krieger $(\Gg, \Lambda; \sigma)$-family} if it satisfies
\begin{enumerate}[labelindent=0pt,labelwidth=\widthof{\ref{itm:CK}},label=(CK), ref=(CK),leftmargin=!]
	\item\label{itm:CK} $S_v = \sum_{\lambda \in v\Lambda^n} S_\lambda S^*_\lambda$ for all $v \in \Lambda^0$ and $n \in \NN^k$.
\end{enumerate}
We prove that these last two conditions are equivalent to \labelcref{itm:cf-range}~and~\labelcref{itm:cf-ck} under a slightly less restrictive requirement than~\labelcref{itm:cf-mult}---we will use the more general form to deal with representations of $\Lambda \bowtie \Gg$ twisted by continuous families of cocycles later in \cref{sec:cocycle_homotopy}.

\begin{prop}\label[prop]{prop:relation_relationships}
	Let $(\Gg, \Lambda)$ be a self-similar action of a groupoid $\Gg$ on a row-finite $k$-graph $\Lambda$ with no sources, let $\Cc \coloneqq \Lambda \bowtie \Gg$ be the associated Zappa--Sz\'ep product category. For $\mu,\nu \in \Lambda$ we have
	\[
	\mu \Cc \cap \nu \Cc = \MCE(\mu,\nu)\Cc.
	\]
	Let $A$ be a $C^*$-algebra, and suppose that $S \colon \Cc \to A$ satisfies \labelcref{itm:cf-source} of \cref{dfn:twisted_lcsc_algebra}, and that $\omega \colon \Cc^2 \to \Uu\Zz \Mm(A)$ satisfies $S_{c_1}S_{c_2} = \omega(c_1,c_2)S_{c_1 c_2}$ for all $(c_1,c_2) \in \Cc^2$. Then
	\begin{enumerate}\setcounter{enumi}{0}
		\item\label{itm:relrel-iii}  $S$ satisfies~\labelcref{itm:TCK3} if and only if it satisfies~\labelcref{itm:cf-range}; and
		\item\label{itm:relrel-iv} $S$ satisfies~\labelcref{itm:CK} and if and only if it satisfies both \labelcref{itm:cf-range} and \labelcref{itm:cf-ck}.
	\end{enumerate}
\end{prop}
\begin{proof}
	
	Fix $\mu,\nu \in \Lambda$. By \labelcref{eq:MCE} we have $\MCE(\mu,\nu) \subseteq r(\mu)\Lambda^{d(\mu) \vee d(\nu)}$. This is a finite set because $\Lambda$ is row-finite.
	Fix $c \in \mu \Cc \cap \nu \Cc$. We can write $c = \mu\mu'g = \nu\nu' h$ for some $\mu', \nu' \in \Lambda$ and $g, h \in \Gg$. Uniqueness of factorisations in $\Cc = \Lambda \bowtie \Gg$ implies that $\mu\mu' = \nu\nu'$. In particular, $d(\mu\mu') = d(\nu\nu') \ge d(\mu) \vee d(\nu)$. By the factorisation property in $\Lambda$, we can write $\mu' = \mu''\tau$ and $\nu' = \nu''\rho$ with $d(\mu'') = (d(\mu) \vee d(\nu)) - d(\mu)$ and $d(\nu'') = (d(\mu) \vee d(\nu)) - d(\nu)$, and then $\mu\mu'' \tau = \nu\nu'' \rho$ with $d(\mu\mu'') = d(\nu\nu'') = d(\mu)\vee d(\nu)$. Uniqueness of factorisations in $\Lambda$ gives $\mu\mu'' = \nu\nu'' \in \MCE(\mu,\nu)$. Hence $c = \mu\mu'g = \mu\mu'' (\tau g) \in \MCE(\mu,\nu)\Cc$. That is, $\mu \Cc \cap \nu \Cc = \MCE(\mu,\nu)\Cc$.

	For~\labelcref{itm:relrel-iii}, first suppose that $S$ satisfies~\labelcref{itm:cf-range}.
	By \labelcref{itm:cf-range}, \cref{rmk:independence_not_needed}, and the first paragraph, we have $S_\mu S^*_\mu S_\nu S_\nu^* = \bigvee_{\lambda \in \MCE(\mu,\nu)} S_\lambda S^*_\lambda$. If $\eta,\zeta \in \MCE(\mu,\nu)$ are distinct, then $d(\eta) = d(\zeta) = d(\mu) \vee d(\nu)$, and so the factorisation property in $\Lambda$ implies that $\eta\Lambda \cap \zeta \Lambda = \varnothing$. Hence~\labelcref{itm:cf-range} implies that $S_\eta S^*_\eta S_\zeta S^*_\zeta = 0$; that is, the projections $S_\lambda S^*_\lambda$, $\lambda \in \MCE(\mu,\nu)$ are mutually orthogonal, so that $\bigvee_{\lambda \in \MCE(\mu,\nu)} S_\lambda S^*_\lambda = \sum_{\lambda \in \MCE(\mu,\nu)} S_\lambda S^*_\lambda$. We deduce that $S$ satisfies~\labelcref{itm:TCK3}. 
	
	Now suppose that $S$ satisfies~\labelcref{itm:TCK3}. Let $c = \mu g \in \Cc$ with $\mu \in \Lambda$ and $g \in \Gg$. Using \cref{lem:twisted_category_algebra_identities} and that $\omega$ takes values in $\Uu\Zz\Mm(A)$, we see that
	\begin{align}\label{eq:principle_ideals}
		\begin{split}
			c \,\Cc &= \mu g\,\Cc \subseteq \mu \,\Cc = \mu g g^{-1} \,\Cc \subseteq c\, \Cc  \quad \text{and} \\
			S_c S^*_c &= \omega(\mu,g)^*\omega(\mu,g) S_\mu S_g S^*_g S^*_\mu = S_\mu S^*_\mu.
		\end{split}
	\end{align}
	
	Fix $c_1, c_2 \in \Cc$ and a finite independent set $F$ such that $c_1 \Cc \cap c_2 \Cc = F\Cc$. Write $c_1 = \mu g$ and $c_2 = \nu h$ with $\mu,\nu \in \Lambda$ and $g,h \in \Gg$, and for each $c \in F$ write $c = \lambda_c g_c$ with $\lambda_c \in \Lambda$ and $g_c \in \Gg$. Let $G \coloneqq \{\lambda_c : c \in F\}$. By the first part of~\labelcref{eq:principle_ideals} we have $\mu_c \Cc \cap \nu_c \Cc = c_1 \Cc \cap c_2 \Cc = F \Cc = G\Cc$, and by the second part we have $S_{c_1} S^*_{c_1} S_{c_2} S^*_{c_2} = S_\mu S^*_\mu S_\nu S^*_\nu$ and $\bigvee_{c \in F} S_c S^*_c = \bigvee_{\lambda \in G} S_\lambda S^*_\lambda$. So it suffices to show that $S_\mu S^*_\mu S_\nu S^*_\nu = \bigvee_{\lambda \in G} S_\lambda S^*_\lambda$.

	By~\labelcref{itm:TCK3}, we have $S_{\mu} S^*_\mu S_\nu S^*_\nu = \sum_{\lambda \in \MCE(\mu,\nu)} S_\lambda S^*_\lambda$.  By the same argument as the second paragraph, for distinct $\lambda,\lambda' \in \MCE(\mu,\nu)$ we have  $\lambda \Lambda \cap \lambda' \Lambda = \varnothing$, so $\MCE(\lambda,\lambda') = \varnothing$. By~\labelcref{itm:TCK3}, we have $S_{\lambda}S_{\lambda}^*S_{\lambda'}S_{\lambda'}^* = 0 = S_{\lambda'}S_{\lambda'}^*S_{\lambda}S_{\lambda}^*$, so $S_{\mu} S^*_\mu S_\nu S^*_\nu = \bigvee_{\lambda \in \MCE(\mu,\nu)} S_\lambda S^*_\lambda$. 
	By uniqueness of factorisation, $\MCE(\mu,\nu) \Lambda = \mu \Lambda \cap \nu \Lambda = G \Lambda$. So for each $\lambda \in G$ there exists $\lambda' \in \MCE(\mu,\nu)$ such that $\lambda \in \lambda' \Lambda$, say $\lambda = \lambda'\tau$, and then $S_{\lambda} S^*_{\lambda} = S_{\lambda'} S_\tau S^*_\tau S^*_{\lambda'} \le S_{\lambda'} S^*_{\lambda'}$. Hence $\bigvee_{\lambda \in G} S_\lambda S^*_\lambda \le \bigvee_{\lambda' \in \MCE(\mu,\nu)} S_{\lambda'} S^*_{\lambda'}$. The same argument gives the reverse inequality, and we deduce that the two are equal. So $S_{\mu}S^*_\mu S_\nu S^*_\nu = \bigvee_{\lambda \in G} S_\lambda S^*_\lambda$ as required.
	
	For~\labelcref{itm:relrel-iv}, first suppose that $S$ satisfies~\labelcref{itm:cf-range} and~\labelcref{itm:cf-ck}. Fix $v \in \Lambda^0$ and $n \in \NN^k$. Fix $c \in v\Cc$ and factorise $c = \lambda g$ with $\lambda \in \Lambda$ and $g \in \Gg$. Since $\Lambda$ is row-finite with no sources, $s(g)\Lambda^n$ is finite and nonempty, say $\tau \in s(g)\Lambda^n$. We have $\lambda g \tau = \lambda (g\rhd \tau)(g\lhd\tau)$. Since self-similar actions are by definition degree preserving, we have $d(\lambda (g\rhd \tau)) = d(\lambda) + n \ge n$, so by the factorisation property, we can write $\lambda (g\rhd \tau) = \nu\eta$ with $d(\nu) = n$ and $d(\eta) = d(\lambda)$. In particular, $\nu \in v\Lambda^n$ and $\lambda g \tau \in c\, \Cc \cap \nu\, \Cc$. That is, $v \Lambda^n$ is exhaustive, and therefore~\labelcref{itm:cf-ck} implies that $S_v = \bigvee_{\mu \in v\Lambda^n} S_\mu S^*_\mu$. If $\mu,\nu \in v \Lambda^n$ are distinct, then the factorisation property implies that $\mu \Lambda \cap \nu \Lambda = \varnothing$, so by~\labelcref{itm:cf-range} the projections $S_\mu S^*_\mu$, $\mu \in v\Lambda^n$ are mutually orthogonal. We conclude that $S_v = \sum_{\mu \in v\Lambda^n} S_\mu S^*_\mu$, so $S$ satisfies~\labelcref{itm:CK}.
	
	Now suppose that $S$ satisfies~\labelcref{itm:CK} and fix a finite exhaustive subset $F \subseteq v\Cc$. Factorise each $c \in F$ as $c = \mu_c g_c$ with $\mu_c \in \Lambda$ and $g_c \in \Gg$, and let $G \coloneqq \{\mu_c : c \in F\}$. Using~\labelcref{eq:principle_ideals} as above, we see that $G\Cc = F\Cc$ and that $\bigvee_{c \in F} S_c S^*_c = \bigvee_{\mu \in G} S_\mu S^*_\mu$. For $\lambda \in v\Lambda$, we have $\lambda \Cc \cap F\Cc \not= \varnothing$ because $F$ is exhaustive, and therefore $\lambda \Cc \cap G \Cc \not= \varnothing$, say $\lambda c = \mu d$. Factorising $c = \alpha g$ and $d = \beta h$ we have $\lambda \alpha g = \mu \beta h$, and then uniqueness of factorisation in $\Lambda \bowtie \Gg$ gives $\lambda\alpha = \mu\beta$. So $\MCE(\lambda,\mu) \not= \varnothing$. That is, $G$ is exhaustive in $\Lambda$. Since $\Lambda$ has no sources, $\Lambda^{\le n} = \Lambda^n$ for all $n$, and so since $S$ satisfies~\labelcref{itm:CK}, it satisfies \cite[Definition~3.3(4)]{RSY03}. So the argument of \cite[Lemma~B.4]{RSY04} shows that $S$ satisfies $S_v = \bigvee_{\mu \in G} S_\mu S^*_\mu$, giving~\labelcref{itm:cf-ck}.
	The argument of \cite[Lemma~7.2]{KPS15} shows that the analogue of \cite[Equation~(7.2)]{KPS15} holds. Applying this with $\mu = \nu$, $\eta = \zeta$, and $n = d(\mu) \vee d(\eta)$, gives~\labelcref{itm:cf-range}.
\end{proof}

\begin{rmk}
	Since~\labelcref{itm:cf-mult} for a cocycle $\sigma$ determines a function $\omega$ as in \cref{prop:relation_relationships} by $\omega(c_1, c_2) = \sigma(c_1, c_2)1_{\Mm(A)}$, \cref{prop:relation_relationships} has useful implications for $C^*$-algebras generated by $\sigma$-twisted representations of categories of the form $\Lambda \bowtie \Gg$ as above. Firstly, it shows that the $C^*$-algebras $\Tt C^*(\Cc, \sigma)$ and $C^*(\Cc, \sigma)$ defined above are canonically isomorphic to the $C^*$-algebras $\Tt C^*(\Gg, \Lambda; \sigma)$ and $C^*(\Gg, \Lambda; \sigma)$ of \cite[Section~7]{MS?}, so our definitions are consistent with the previous literature. Secondly, and as a result, it gives us access to the results of \cite[Section~7]{MS?}. In particular, if $\Cc = \Lambda \bowtie \Gg$ is the Zappa--Sz\'ep product associated to a self-similar action of a groupoid on a row-finite $k$-graph with no sources and $\sigma$ is a $2$-cocycle on $\Cc$, then \cite[Proposition~7.7]{MS?} implies that there are injective homomorphisms $\Tt C^*(\Lambda, \sigma) \to \Tt C^*(\Cc, \sigma)$ and $C^*(\Lambda, \sigma) \to C^*(\Cc, \sigma)$ carrying generators to generators. Thirdly, we saw in the proof of \cref{prop:relation_relationships} that for $\mu,\nu \in \Lambda$, the set $\MCE(\mu,\nu)$ is an independent set generating $\mu\Cc \cap \nu\Cc$; since $\mu \vee \nu$ denotes any choice of such a set, we can always take $\mu \vee \nu = \MCE(\mu,\nu)$ and so \cref{prop:relation_relationships}\labelcref{itm:relrel-iii} implies that
	\begin{equation}\label{eq:range_prod_sum}
		S_\mu S^*_\mu S_\nu S^*_\nu = \sum_{\lambda \in \mu \vee \nu} S_\lambda S^*_\lambda
	\end{equation}
	for all $\mu,\nu \in \Lambda$ and any $\sigma$-twisted representation $S$ of $\Cc$.
\end{rmk}

\begin{example}
	Let $\Cc = \Lambda$ be a row-finite $k$-graph with no sources and let $S \colon \Lambda \to A$ be a $\sigma$-twisted representation of $\Lambda$ in a $C^*$-algebra $A$. Let $\Gg = \Lambda^0$, the trivial groupoid with unit space $\Lambda^0$ in which the only elements are the units. Then  $\Gg$ trivially acts self-similarly on $\Lambda$, and $\Lambda \bowtie \Gg \cong \Lambda$. \Cref{prop:relation_relationships} then implies that $C^*(\Lambda \bowtie \Gg, \sigma)$ is the twisted $k$-graph $C^*$-algebra $C^*(\Lambda, \sigma)$ of \cite{KPS15}.
\end{example}

\subsection{Maps induced by inclusions of subcategories}
We are interested in which inclusions of categories induce homomorphisms between twisted $C^*$-algebras. As \cref{example:concordance_counterexample} below demonstrates, it is not typically true that an inclusion of one category in another induces a $*$-homomorphism, even of Toeplitz algebras. To identify the obstruction we introduce the following terminology.

\begin{dfn}\label[dfn]{dfn:concordance}
	Let $\Xx$ be a finitely aligned left-cancellative small category. Let $\Cc$ be a subcategory of $\Xx$. We say that $\Cc$ is a \emph{concordant} subcategory if for every $c_1,c_2 \in \Cc$ with $c_1\Xx \cap c_2\Xx \ne \varnothing$ there is a finite independent set $F \subseteq \Cc$ generating $c_1 \Cc \cap c_2 \Cc$ such that for all $x_1,x_2 \in \Xx$ satisfying $c_1 x_1 = c_2 x_2$ there exist $c_1a_1 = c_2 a_2 \in F$ and $y \in \Xx$ such that the diagram
	\[
	\begin{tikzcd}
		& \bullet \\
		\bullet && \bullet & \bullet \\
		& \bullet
		\arrow["{c_1}"', from=1-2, to=2-1]
		\arrow["{a_1}"', dashed, from=2-3, to=1-2]
		\arrow["{a_2}", dashed, from=2-3, to=3-2]
		\arrow["{x_1}"', bend right, from=2-4, to=1-2]
		\arrow["y"', dashed, from=2-4, to=2-3]
		\arrow["{x_2}", bend left, from=2-4, to=3-2]
		\arrow["{c_2}", from=3-2, to=2-1]
	\end{tikzcd}
	\]
	commutes.
\end{dfn}

\begin{rmk}
	\emph{A priori} the $a_1$ and $a_2$ in \cref{dfn:concordance} need not belong to $\Cc$ as long as $c_1 a_1 = c_2 a_2$ does; but since $c_1 a_1 = c_2 a_2 \in F \subseteq c_1 \Cc \cap c_2 \Cc$, there exist $a_1', a_2' \in \Cc$ such that $c_1 a_1' = c_1 a_1 = c_2 a_2 = c_2 a_2'$, and then left cancellativity gives $a_1 = a_1' \in \Cc$ and $a_2 = a_2' \in \Cc$.
\end{rmk}

\begin{rmk}
	If $\Cc$ is a concordant subcategory of a finitely aligned left-cancellative small category $\Xx$, then $\Cc$ is itself finitely aligned and left-cancellative.
\end{rmk}

\begin{ntn}
	If $\Cc$ is a subcategory of $\Xx$ and $\sigma \colon \Xx^2 \to \TT$ is a $2$-cocycle then we abuse notation and also write $\sigma$ for the restriction $\sigma|_{\Cc^2} \colon \Cc^2 \to \TT$, which is itself a $2$-cocycle on $\Cc$.
\end{ntn}

\begin{prop}\label[prop]{prop:toeplitz_map}
	Let $\Xx$ be a finitely aligned left-cancellative small category and suppose that $\Cc$ is a concordant subcategory of $\Xx$. Fix a $2$-cocycle $\sigma \colon \Xx^2 \to \TT$.
	Let $t^\Cc \colon \Cc \to \Tt C^*(\Cc, \sigma)$ and $t^{\Xx} \colon \Xx \to \Tt C^*(\Xx, \sigma)$ be universal representations.
	There is a unique $*$-homomorphism $\Phi \colon \Tt C^*(\Cc, \sigma) \to \Tt C^*(\Xx , \sigma)$ such that $\Phi(t_c^{\Cc}) = t_c^{\Xx}$ for all $c \in \Cc$. If, for every $v \in \Cc^0$, every finite exhaustive set $F \subseteq v \Cc$ is also exhaustive in $v\Xx$, then $\Phi$ descends to a $*$-homomorphism $\ol{\Phi} \colon C^*(\Cc, \sigma) \to C^*(\Xx, \sigma)$ such that $\ol{\Phi}(s_c^{\Cc}) = s_c^{\Xx}$ for all $c \in \Cc$.
\end{prop}
\begin{proof}
	
	For the first statement it suffices to show that $t^{\Xx}|_{\Cc} \colon \Cc \to \Tt C^*(\Xx, \sigma)$ satisfies \labelcref{itm:cf-mult}--\labelcref{itm:cf-range} of \cref{dfn:twisted_lcsc_algebra} for $(\Cc,\sigma)$. That \labelcref{itm:cf-mult}~and~\labelcref{itm:cf-source} hold is immediate: these relations for elements of $\Cc$ are identical to the same relations for those elements regarded as elements of $\Xx$.
	
	For \labelcref{itm:cf-range} suppose that $c_1,c_2 \in \Cc$. First suppose that $c_1 \Xx \cap c_2 \Xx = \varnothing$. Then $c_1 \Cc \cap c_2 \Cc = \varnothing$ and so  \labelcref{itm:cf-range} holds since $t_{c_1}^{\Cc}t_{c_1}^{\Cc *} t_{c_2}^{\Cc} t_{c_2}^{\Cc *} = 0$.
	Now suppose that  $c_1 \Xx \cap c_2 \Xx \ne \varnothing$. Since $\Cc$ is concordant in $\Xx$ there exists a finite independent set $F \subseteq \Cc$ generating $c_1 \Cc\cap c_2 \Cc$ satisfying the condition of \cref{dfn:concordance}. We claim that $F$ generates $ c_1 \Xx \cap c_2 \Xx$. If $x \in c_1 \Xx \cap c_2 \Xx$, then there exist $x_1,x_2 \in \Xx$ such that $x = c_1 x_1 = c_2 x_2$. Since $F$ satisfies the condition of \cref{dfn:concordance}, there exist $c_1 a_1 = c_2 a_2 \in F$ and $y \in \Xx$ such that $x_1 = a_1 y$ and $x_2 = a_2 y$. Hence,
	\[
	x = c_1 x_1 = c_1 a_1 y \in F \Xx.
	\]
	That is, $F$ generates $ c_1 \Xx \cap c_2 \Xx$.  \cref{rmk:independence_not_needed} gives $t^\Xx_{c_1} (t^\Xx_{c_1})^* t^\Xx_{c_2} (t^\Xx_{c_2})^* = \bigvee_{c \in F} t^\Xx_{c} (t^\Xx_{c})^*$, and since $F \subseteq \Cc$, this gives~\labelcref{itm:cf-range} for the representation $t^{\Xx}|_{\Cc}$ of $(\Cc,\sigma)$.
	
	For the second statement let $s^{\Xx} \colon \Xx \to C^*(\Xx,\sigma)$ be a universal covariant representation. It suffices to show that $s|_{\Cc} \colon \Cc \to C^*(\Xx,\sigma)$ satisfies \labelcref{itm:cf-ck} of \cref{dfn:twisted_lcsc_algebra}; but this follows from the hypothesis that finite exhaustive sets in $\Cc$ are exhaustive in $\Xx$.
\end{proof}

The next example demonstrates that inclusions of the form $\Cc \hookrightarrow \Dd \bowtie \Cc$ are not always concordant, even for monoids.

\begin{example}\label[example]{example:concordance_counterexample}
	Define an action $\la$ of $\FF_2^{+} = \langle a, b\rangle$ on $\NN$ by $w \la n = n$ for all $w \in \FF_2^+$, and $n \in \NN$ and a right action of $\NN$ on $\FF_2^{+}$ by $w \ra n = a^{|w|}$ for all $w \in \FF_n^+$ and $n \in \NN \setminus \{0\}$ and $w \ra 0 = w$.
	Then $(\FF_2^+,\NN)$ is a matched pair. Since $w \la \cdot : \NN \to \NN$ is injective for each $w \in \FF_2^+$, \cite[Lemma~3.15]{MS?} implies that $\Xx \coloneqq \NN \bowtie \FF_2^{+}$ is left cancellative.
	
	Recall that for $w,u \in \FF_2^+$, if $w \FF_2^+ \cap u \FF_2^+$ is nonempty, then either $w = uw'$ or $u = w u'$; we then denote the unique minimal common extension by $w \vee u$.
	
	To see that $\Xx$ is finitely aligned fix $n \in \NN$ and $ w,u \in \FF_2^+$. We first claim that
	\begin{equation}\label{eq:bad_ideals}
		nw \Xx \cap nu \Xx =
		\begin{cases*}
			n (w \vee u) \Xx \cup (n+1)a^{\max\{|w|,|u|\}} \Xx & if $w \FF_2^+ \cap u \FF_2^+ \ne \varnothing$,\\
			(n+1)a^{\max\{|w|,|u|\}} \Xx & otherwise.
		\end{cases*}
	\end{equation}
	For $\supseteq$ note that if $w \FF_2^+ \cap u \FF_2^+ \ne \varnothing$, then $n(w \vee u) \in nw\Xx \cap nu \Xx$. For any $v \in \FF_2^+$ we have $nv1 = n(v \la 1)(v \ra 1)= (n+1) a^{|v|}$. In particular, $(n+1)a^{\max\{|w|,|u|\}} \in  nw\Xx \cap nu \Xx$. Hence, the right-hand side of \eqref{eq:bad_ideals} is contained in the left-hand side.
	
	Now suppose that $x \in nw \Xx \cap nu \Xx$. Then $x = nwkv = nuk'v'$ for some $k,k' \in \NN$ and $v,v' \in \FF_2^{+}$. So
	\[
	(n+k)(w \ra k) v = nwkv = nuk'v' = (n+k') (u \ra k') v'.
	\]
	By uniqueness of factorisation $k = k'$. If $k = 0$, then $x = nwv = nuv'$ so $wv = uv' \in w \FF_2^+ \cap u \FF_2^+$ and $x \in n(w \vee u) \Xx$. If $k \ge 1$, then $x = (n+1)a^{|w|}(k-1)v = (n+1)a^{|u|}(k'-1)v' \in (n+1) a^{\max\{|w|,|u|\}}\Xx$. So \eqref{eq:bad_ideals} holds as claimed.
	
	Now suppose that $nw \Xx \cap m u \Xx$ is nonempty, say $x \in nw \Xx \cap m u \Xx$. Without loss of generality, $m \ge n$, so $m = n + n'$ for some $n' > 0$. If $m = n$, then \eqref{eq:bad_ideals} establishes that $nw\Xx \cap m u \Xx$ is a finite union of principal right ideals. So suppose that $m > n$.
	Then $x = nw k v = (n+n')u k' v'$ for some $k,k' \in \NN$ and $v,v' \in \FF_2^+$. Uniqueness of factorisations implies that $k \ge n'$, so $x = (n+n')a^{|w|} (k-n')v' = (n+n') u k'v'$. It follows that $nw \Xx \cap mu \Xx = m a^{|w|} \Xx \cap mu \Xx$ which is an intersection of the form covered by \eqref{eq:bad_ideals}. Hence, $\Xx$ is finitely aligned.

	The inclusion $\FF_2^+ \hookrightarrow \Xx$ is not concordant since the ideal $a \Xx \cap b\Xx$ contains $b1 = 1a = a1$, but $a \FF_2^+ \cap b \FF_2^+ = \varnothing$.
	On the level of Toeplitz algebras, in $\Tt C^*(\FF_2^+, 1)$ we have $t_a t_a^* t_b t_b^* = 0$, while in $\Tt C^*(\Xx, 1)$ the element $t_{1a}t_{1a}^{ *}$ is a nonzero subprojection of $t_a t_a^{ *} t_b t_b^{ *}$. Since $*$-homomorphisms preserve orthogonality, there is no $*$-homomorphism $\Phi \colon \Tt C^*(\FF_2^+, 1) \to \Tt C^*(\Xx, 1)$ such that $\Phi \circ t^{\FF_2^+} = t^{\Xx}$.
\end{example}

\begin{lem}\label[lem]{lem:self_similar_concordance}
	Fix $k \ge 0$. Suppose that $\Gg$ is a discrete groupoid acting self-similarly on a finitely aligned $(k+1)$-graph $\Lambda$. Let $\Gamma = \Lambda^{\NN^k}$ and $E^* = \Lambda^{\NN e_{k+1}}$, so that $\Lambda = E^* \bowtie \Gamma$ as in \cref{ex:k-graph_zs}. Then $\Gamma \bowtie \Gg \hookrightarrow \Lambda \bowtie \Gg$ is concordant.
\end{lem}

\begin{proof}
	
	Fix $\mu_1g_1, \mu_2 g_2 \in \Gamma \bowtie \Gg$. Suppose that $x_1,x_2 \in \Lambda \bowtie \Gg$ satisfy $\mu_1 g_1 x_1 = \mu_2 g_2 x_2$. For $i = 1,2$ there exist unique $\nu_i \in \Gamma$, $\alpha_i \in E^*$ and $h_i \in \Gg$ such that $x_i = \nu_i \alpha_i h_i$. We have
	\begin{equation}\label{eq:inverses_cancelling}
		\mu_i g_i \nu_i (g \ra \nu_i)^{-1} = \mu_i (g_i \la \nu_i )(g_i \ra \nu_i)(g_i \ra \nu_i)^{-1} = \mu_i (g_i \la \nu_i ).
	\end{equation}
	Since $\mu_1 g_1 x_1 = \mu_2 g_2 x_2$, we have
	\[
	\mu_1 (g_1 \la \nu_1) ((g_1 \ra \nu_1) \la \alpha_1)(g_1 \ra \nu_1\alpha_1)h_1 = \mu_2 (g_2 \la \nu_2) ((g_2 \ra \nu_2) \la \alpha_2)(g_2 \ra \nu_2\alpha_2)h_2.
	\]
	Uniqueness of factorisation in $\Lambda \bowtie \Gg$ implies that $\mu_1 (g_1 \la \nu_1) ((g_1 \ra \nu_1) \la \alpha_1) = \mu_2 (g_2 \la \nu_2) ((g_2 \ra \nu_2) \la \alpha_2)$, so uniqueness of factorisations in $\Lambda = E^* \bowtie \Gamma$ and that the action of $\Gg$ on $\Lambda$ is degree preserving (\cref{dfn:self-similar_action}) implies that $\mu_1(g_1 \la \nu_1) = \mu_2 (g_2 \la \nu_2)$.
	It follows from~\eqref{eq:inverses_cancelling} that $\mu_1 g_1 \nu_1 (g \ra \nu_1)^{-1} = \mu_2 g_2 \nu_2 (g \ra \nu_2)^{-1} \in \mu_1 g_1 (\Gamma \bowtie \Gg) \cap \mu_2 g_2 (\Gamma \bowtie \Gg)$.
	
	Since $\Gamma \bowtie \Gg$ is finitely aligned (\cref{lem:groupoid_gives_finitely_aligned}), by the discussion following \labelcref{eq:finitely_aligned} there is a finite independent set $F \subseteq \Gamma \bowtie \Gg$ generating $\mu_1 g_1 (\Gamma \bowtie \Gg) \cap \mu_2 g_2 (\Gamma \bowtie \Gg)$. Hence, there exists $f \in F$ such that $\mu_1 g_1 \nu_1 (g \ra \nu_1)^{-1} = \mu_2 g_2 \nu_2 (g \ra \nu_2)^{-1} \in f (\Gamma \bowtie \Gg)$. So there exist $a_1,a_2, y \in \Gamma \bowtie \Gg$ such that $f = \mu_1 g_1 a_1 = \mu_2 g_2 a_2$ and
	\begin{equation}\label{eq:concordant_1}
		\mu_1 g_1 \nu_1 (g \ra \nu_1)^{-1} = \mu_1 g_1 a_1 y = fy= \mu_2 g_2 a_2 y = \mu_2 g_2 \nu_2 (g \ra \nu_2)^{-1}.
	\end{equation}
	We have
	\begin{equation}\label{eq:concordant_2}
		\mu_i g_i x_i  = \mu_i g_i \nu_i (g \ra \nu_i)^{-1}(g_i \ra \nu_i) \alpha_i h_i = \mu_i g_i a_i y (g_i \ra \nu_i) \alpha_i h_i \in F(\Lambda \bowtie \Gg).
	\end{equation}
	Now consider the following diagram.
	\[\begin{tikzcd}[ampersand replacement=\&,column sep=large]
		\& \bullet \& \bullet \\
		\bullet \& \bullet \& \bullet \& \bullet \\
		\& \bullet \& \bullet
		\arrow["{\mu_1g_1}"', from=1-2, to=2-1]
		\arrow["{\nu_1}"', from=1-3, to=1-2]
		\arrow["{g_1 \ra \nu_1}"', from=1-3, to=2-3]
		\arrow["{a_1}", from=2-2, to=1-2]
		\arrow["{a_2}"', from=2-2, to=3-2]
		\arrow["y"', from=2-3, to=2-2]
		\arrow["{\alpha_1 h_1}"', from=2-4, to=1-3]
		\arrow["{\alpha_2h_2}", from=2-4, to=3-3]
		\arrow["{\mu_2 g_2}", from=3-2, to=2-1]
		\arrow["{g_2 \ra \nu_2}", from=3-3, to=2-3]
		\arrow["{\nu_2}", from=3-3, to=3-2]
	\end{tikzcd}\]
	The outside hexagon commutes by assumption. The left-hand triangle commutes by definition of $a_1$ and $a_2$. Cancelling $\mu_1 g_1$ from the left side of the first equality of \labelcref{eq:concordant_1} and then multiplying on the right by $g_1 \ra \nu_1$ shows that the top rectangle commutes. Similarly, the bottom rectangle commutes. Cancelling \labelcref{eq:concordant_1} from the left of the second equality of \labelcref{eq:concordant_2} shows that the right-hand triangle commutes. So the diagram commutes and hence $\Gamma \bowtie \Gg$ is concordant in $\Lambda \bowtie \Gg$.
\end{proof}

\begin{lem}\label[lem]{lem:k_graph_exhaustive}
	Fix $k \ge 0$. Let $\Lambda$ be a locally convex $(k+1)$-graph. Let $\Gamma = \Lambda^{\NN^k}$ and $E^* = \Lambda^{\NN e_{k+1}}$, so that $\Lambda = E^* \bowtie \Gamma$ as in \cref{ex:k-graph_zs}. Then every finite exhaustive set in $\Gamma$ is exhaustive in $\Lambda$.
\end{lem}
\begin{proof}
	Fix $v \in \Lambda^0$ and a finite exhaustive set $F \subseteq v\Gamma$.  Fix $\lambda \in v \Lambda$. Let $ n \coloneqq (d(\lambda)_1,\ldots,d(\lambda)_k) \in \NN^k$ and $n' \coloneqq d(\lambda)_{k+1}e_{k+1}$. Let $m \coloneqq \bigvee_{\mu \in F} d(\mu)$.
	
	By \cite[Remarks~3.2]{RSY03} $s(\lambda) \Gamma^{\le m}$ is nonempty, so fix $\tau \in s(\lambda) \Gamma^{\le m}$. Since $\lambda \in \Lambda^{n + n'} \subseteq \Lambda^{\le n + n'}$ Lemma~3.6 of \cite{RSY03} implies that $\lambda \tau \in \Lambda^{\le n+n'+m}$. By \labelcref{eq:le_factorisation} we have $\lambda \tau\in \Lambda^{\le m+n} \Lambda^{\le n'}$. Since $m+n \in \NN^k$ we have $\Lambda^{\le m+n} = \Gamma^{\le m+ n}$. So by \labelcref{eq:le_factorisation} again, there exist $\tau' \in \Gamma^{\le m+ n}$ and $\alpha \in \Lambda^{n'}$ such that $ \lambda \tau=  \tau' \alpha$. Since $F$ is exhaustive in $\Gamma$, there exists $\mu \in F$ such that $\mu \Gamma \cap  \tau' \Gamma \ne \varnothing$. By yet another application of \labelcref{eq:le_factorisation} we have $\Gamma^{\le m+ n} = \Gamma^{\le m} \Gamma^{\le n}$, and so $\tau' = \eta \xi$ for some $\eta \in \Gamma^{\le m}$ and $\xi \in \Gamma^{\le n}$. As $\mu \Gamma \cap  \tau' \Gamma \ne \varnothing$ we have $\mu \Gamma \cap \eta \Gamma \ne \varnothing$ so by the second statement of \cref{lem:le_extensions_equal} we have $\eta \in \mu \Lambda$.  Hence,
	$
	\lambda \tau = \tau' \alpha = \eta \xi \alpha \in \mu \Lambda,
	$
	so $\mu \Lambda \cap \lambda \Lambda \ne \varnothing$. That is, $F \subseteq v \Lambda$ is exhaustive.
\end{proof}

\begin{lem}\label[lem]{lem:self-similar_exhaustive}
	Fix $k \ge 0$. Suppose that $\Gg$ is a discrete groupoid acting self-similarly on a locally convex $(k+1)$-graph $\Lambda$.  Let $\Gamma = \Lambda^{\NN^k}$ and $E^* = \Lambda^{\NN e_{k+1}}$, so that $\Lambda = E^* \bowtie \Gamma$ as in \cref{ex:k-graph_zs}. Then every finite exhaustive set in $\Gamma \bowtie \Gg$ is exhaustive in $\Lambda \bowtie \Gg$.
\end{lem}
\begin{proof}
	Fix a finite exhaustive set $F \subseteq v (\Gamma \bowtie \Gg)$. Then for every $\mu' \in v\Gamma \subseteq \Gamma \bowtie \Gg$, there exist $\mu g \in F$ and $\nu_1h_1,\nu_2h_2 \in \Gamma \bowtie \Gg$ such that $\mu' \nu_1h_1 = \mu  g \nu_2h_2 = \mu (g \la \nu_2)(g \ra \nu_2) h_2$. By uniqueness of factorisation in $\Gamma \bowtie \Gg$ we have $\mu' \nu_1 = \mu (g \la \nu_2)$ in $\Gamma$; that is $\mu' \Gamma \cap \mu \Gamma \ne \varnothing$. In other words, $F' \coloneqq \{\mu \in \Gamma \mid \mu
	\Gg \cap F \ne \varnothing\}$ is exhaustive in $\Gamma$, so by \cref{lem:k_graph_exhaustive} $F'$ is exhaustive in $\Lambda$.
	
	Now fix  $\nu h \in \Lambda \bowtie \Gg$.
	Since $F' $ is exhaustive in $\Lambda$, there exist $\mu g \in F$ and $\mu',\nu' \in \Lambda$ such that $\mu \mu' = \nu \nu'$. We have
	\begin{align*}
		\nu h (h^{-1} \la \nu') (h^{-1} \ra \nu') = \nu h h^{-1} \nu'
		= \nu \nu ' = \mu \mu' = \mu g g^{-1} \mu',
	\end{align*}
	so $\nu h (\Lambda \bowtie \Gg) \cap \mu g (\Lambda \bowtie \Gg) \ne \varnothing$. That is, $F$ is exhaustive in $\Lambda \bowtie \Gg$.
\end{proof}

\begin{thm}\label[thm]{thm:self-similar_inductive_maps}
	Fix $k \ge 0$. Suppose that $\Gg$ is a discrete groupoid acting self-similarly on a finitely aligned $(k+1)$-graph $\Lambda$.  Let $\Gamma = \Lambda^{\NN^k}$ and $E^* = \Lambda^{\NN e_{k+1}}$, so that $\Lambda = E^* \bowtie \Gamma$ as in \cref{ex:k-graph_zs}. Let $\sigma$ be a $\TT$-valued $2$-cocycle on $\Lambda \bowtie \Gg$.
	There is a unique $*$-homomorphism $\Phi \colon \Tt C^*(\Gamma \bowtie \Gg, \sigma) \to \Tt C^*(\Lambda \bowtie \Gg, \sigma) $ such that $\Phi(t_{\gamma g}^{\Gamma \bowtie \Gg}) = t_{\gamma g}^{\Lambda \bowtie \Gg}$ for all $\gamma g \in \Gamma \bowtie \Gg$.
	If $\Lambda$ is locally convex, then $\Phi$ descends to a $*$-homomorphism $\ol{\Phi} \colon  C^*(\Gamma \bowtie \Gg, \sigma) \to  C^*(\Lambda \bowtie \Gg, \sigma) $ such that $\ol{\Phi}(s_{\gamma g}^{\Gamma \bowtie \Gg}) = s_{\gamma g}^{\Lambda \bowtie \Gg}$ for all $\gamma g \in \Gamma \bowtie \Gg$.
\end{thm}

\begin{proof}
	By \cref{lem:sub_kgraph_fin_align}, $\Gamma$ is finitely aligned. Hence, \cref{lem:groupoid_gives_finitely_aligned} implies that $\Lambda \bowtie \Gg$ is left cancellative and finitely aligned. By \cref{lem:self_similar_concordance} the inclusion $\Gamma \bowtie \Gg \hookrightarrow \Lambda \bowtie \Gg$ is concordant, so \cref{prop:toeplitz_map} gives the desired homomorphism $\Phi \colon \Tt C^*(\Gamma \bowtie \Gg, \sigma) \to \Tt C^*(\Lambda \bowtie \Gg, \sigma)$. If $\Lambda$ is locally convex, then \cref{lem:self-similar_exhaustive} says that every finite exhaustive set in $\Gamma \bowtie \Gg$ is also exhaustive in $\Lambda \bowtie \Gg$. So \cref{prop:toeplitz_map} implies that $\Phi$ descends to a $*$-homomorphism $\ol{\Phi} \colon  C^*(\Gamma \bowtie \Gg, \sigma) \to  C^*(\Lambda \bowtie \Gg, \sigma) $ such that $\ol{\Phi}(s_{\gamma g}^{\Gamma \bowtie \Gg}) = s_{\gamma g}^{\Lambda \bowtie \Gg}$ for all $\gamma g \in \Gamma \bowtie \Gg$.
\end{proof}

We are interested in when the $*$-homomorphism $\ol{\Phi}$ of \cref{thm:self-similar_inductive_maps} is injective. To this end, we extend the terminology of Yusnitha~\cite{YusPhD} to our setting.

\begin{dfn}
	A self-similar action of a discrete groupoid $\Gg$ on a $k$-graph $\Lambda$ is \emph{jointly faithful} if for each $v \in \Lambda^0$ and each $n \in \NN^{k}$ there exists $\lambda \in \Lambda^n$ such that the map $g \mapsto (g \la \lambda, g \ra \lambda)$ is injective on $v \Gg v$.
\end{dfn}

\begin{rmk}\label[rmk]{rmk:amenability}
	For the next result, we need to assume that $\Gg$ is amenable, so that we can invoke \cite[Proposition~7.7]{MS?} to see that $C^*(\Gg, \sigma)$ embeds in $C^*(\Lambda \bowtie \Gg, \sigma)$. As in \cite[Chapter~9]{Wil19}, by \emph{amenable} we mean topologically amenable in the sense of \cite{Renault80} (see \cite[Definition~9.3]{Wil19}), and throughout this remark we adopt the notation of \cite[Chapter~9]{Wil19}. In our situation where $\Gg$ is discrete, topological amenability is equivalent to the condition that each isotropy group $x \Gg x$ is amenable as a discrete group. 
	
	To see why, first note that every discrete groupoid $\Gg$ is equivalent to a discrete group bundle whose fibres are isotropy groups of $\Gg$: fix any $X \subseteq \Gg^0$ that meets each orbit once in the sense that $x \Gg y = \varnothing$ for distinct $x,y \in \Gg$ and $r(\Gg X) = \Gg^0$; then by \cite[Example~2.39]{Wil19} applied with $A = X$ and $B = \Gg^0$, the space $\Gg X$ is a groupoid equivalence between $X \Gg X = \bigsqcup_{x \in X} x \Gg x$ and $\Gg$. Next, recall that topological amenability is preserved by equivalence \cite[Theorem~2.2.17]{A-DR00}, and so $\Gg$ is topologically amenable if and only if $X \Gg X$ is topologically amenable. Counting measures constitute a Haar system $(\lambda^x)_{x \in X}$ for $X \Gg X$, and given a net $(f_i)_{i \in I}$ of functions as in \cite[Definition~9.3]{Wil19}, for each $x \in X$ the net $(f_i|_{x \Gg x})_{i \in I}$ determines an approximate invariant mean on the group $x \Gg x$, so $x \Gg x$ is amenable. 
	
	Conversely if each $x \Gg x$ is amenable, we can fix an approximate invariant mean $\{f^x_n\}_{n \in \NN}$ on $x \Gg x$ (each $x \Gg x$ is countable, so sequences suffice). Consider the directed set $F(X)$ of finite subsets of $X$ under set containment. Then $\NN \times F(X)$ is a directed set in lexicographic order. For $(n, A) \in \NN \times F(X)$, the function $f_{n, A} = \sum_{x \in A} f^x_n$ belongs to $C_c(X \Gg X)$, and the net $(f_{n,A})_{(n, A) \in \NN \times F(X)}$ satisfies the requirements of \cite[Definition~9.3]{Wil19}; so $X \Gg X$ is topologically amenable.
\end{rmk}

In the statement of the following corollary, we invoke \cref{thm:self-similar_inductive_maps}, which applies only when $\Lambda$ is locally convex, to obtain the map $\ol{\Phi}$. So it is important to recall that $k$-graphs with no sources are always locally convex.

\begin{cor}\label[cor]{cor:subalg}
	Consider a jointly faithful self-similar action of a discrete amenable groupoid $\Gg$ on a row-finite $(k+1)$-graph $\Lambda$ with no sources.  Let $\Gamma = \Lambda^{\NN^k}$ and $E^* = \Lambda^{\NN e_{k+1}}$, so that $\Lambda = E^* \bowtie \Gamma$ as in \cref{ex:k-graph_zs}, and let $\sigma$ be a $\TT$-valued $2$-cocycle on $\Lambda \bowtie \Gg$. Then the $*$-homomorphism $\ol{\Phi} \colon  C^*(\Gamma \bowtie \Gg, \sigma) \to  C^*(\Lambda \bowtie \Gg, \sigma) $ of \cref{thm:self-similar_inductive_maps} is injective.
\end{cor}
\begin{proof}
	By \cite[Proposition~7.7]{MS?} the $*$-homomorphism $\iota_\Gg^{\Gamma} \colon C^*(\Gg,\sigma) \to C^*(\Gamma \bowtie \Gg, \sigma)$ satisfying $\iota_{\Gg}^{\Gamma} (u_g) = s_g$ and the $*$-homomorphism $\iota_\Gg^{\Lambda} \colon C^*(\Gg,\sigma) \to C^*(\Lambda \bowtie \Gg, \sigma)$ satisfying $\iota_{\Gg}^{\Lambda} (u_g) = s_g$ are injective. Since $\ol{\Phi} \circ \iota_{\Gg}^{\Gamma} = \iota_{\Gg}^{\Lambda}$, the Gauge-Invariant Uniqueness Theorem \cite[Corollary~7.10]{MS?} implies that $\ol{\Phi}$ is injective.
\end{proof}

\section{Homotopy of cocycles}
\label{sec:cocycle_homotopy}

We are interested in the effect on $C^*(\Cc, \sigma)$ of continuous variations in $\sigma$. To this end we introduce the following notion.
\begin{dfn}
	Let $\II \coloneqq [0,1]$. A \emph{homotopy of 2-cocycles} on a small category $\Cc$ is a continuous map $\Sigma \colon \II \times \Cc^2 \to \TT$, $(t,c_1,c_2) \mapsto \Sigma_t(c_1,c_2)$ whose restriction $\Sigma_t$ to $\{t\} \times \Cc^2$ is a 2-cocycle on $\Cc$ for each $t \in \II$. We say that the $2$-cocycles $\Sigma_0$ and $\Sigma_1$ are \emph{homotopic}. 	For each $(c_1,c_2) \in \Cc^2$ we let $\Sigma_{\bullet}(c_1,c_2)$ denote the continuous function $t \mapsto \Sigma_t(c_1,c_2)$ on $\II$.
\end{dfn}

\subsection{Twisting representations by homotopies of \texorpdfstring{\boldmath$2$}{2}-cocycles}

We consider representations of categories twisted by homotopies of $2$-cocycles. The idea is that such a representation should be thought of as a bundle of representations of the category twisted by continuously varying $2$-cocycles.

\begin{dfn}\label{dfn:twisted_lscs_algebra_homotopy}
	Let $\Cc$ be a finitely aligned left-cancellative small category and let $\Sigma$ be a homotopy of $2$-cocycles for $\Cc$. A \emph{$\Sigma$-twisted representation} of $\Cc$ in a $C^*$-algebra $A$ is a pair $(S,\CImap)$ consisting of a map $S \colon \Cc \to A$, $c \mapsto S_{c}$ such that each $S_c$ is a partial isometry, together with a $*$-homomorphism $\CImap_v \colon C(\II) \to A$ for each $v \in \Cc^0$,  such that $S$ satisfies \labelcref{itm:cf-source} and \labelcref{itm:cf-range} of \cref{dfn:twisted_lcsc_algebra}, the $\CImap_v$ have mutually orthogonal images, and	
	\begin{enumerate}[labelindent=0pt,labelwidth=\widthof{\ref{itm:cfc-mult}},label=(IR\arabic*), ref=(IR\arabic*),leftmargin=!]
		\item \label{itm:cfc-one}
		$\CImap_v(1_\II) = S_v$  for all $v \in \Cc^0$,
		\item \label{itm:cfc-commute}
		$\CImap_{r(c)}(f) S_c = S_c \CImap_{s(c)}(f)$ for all $c \in \Cc$ and $f \in C(\II)$, and
		\item \label{itm:cfc-mult}	$S_{c_1} S_{c_2} = \CImap_{r(c_1)}\big( \Sigma_{\bullet}(c_1,c_2) \big) S_{c_1 c_2}$ for all $(c_1,c_2) \in \Cc^2$.
	\end{enumerate}
	We say that $(S,\CImap)$ is \emph{covariant} if $S$ satisfies \labelcref{itm:cf-ck}.
	
	Given a $\Sigma$-twisted representation $(S,\CImap)$ of $\Cc$ in $A$, the \emph{$C^*$-algebra generated by $(S,\CImap)$}, denoted $C^*(S,\CImap)$, is the $C^*$-subalgebra of $A$ generated by $\{S_c \mid c \in \Cc\} \cup \bigcup_{v \in \Cc^0} \CImap_v (C(\II))$. The $ C^*$-algebra $\Tt C^*_{\II}(\Cc, \Sigma)$ is the universal $C^*$-algebra generated by a $\Sigma$-twisted representation $(\mathfrak{t}, \frakCImap)$ in the sense that $\Tt C^*_{\II}(\Cc, \Sigma) = C^*(\mathfrak{t}, \frakCImap)$ and if $(S,\CImap)$ is a  $\Sigma$-twisted representation in a $C^*$-algebra $A$, then there is a unique $*$-homomorphism $\Phi \colon \Tt C^*_{\II}(\Cc, \Sigma) \to A$ such that $S = \Phi \circ \mathfrak{t}$ and $\CImap_v = \Phi \circ \frakCImap_v$ for all $v \in \Cc^0$. The $C^*$-algebra $C^*_{\II}(\Cc, \Sigma)$ is the universal $C^*$-algebra generated by a covariant $\Sigma$-twisted representation $(\mathfrak{s},\frakCImap)$.
\end{dfn}

\begin{rmk}
	We chose the letters $\CImap$ and $\frakCImap$ for the collections of maps in the preceding definition because ``$z$'' suggests central elements---see \labelcref{lem:C(I)-algebra}. Although $\frakCImap$ denotes the maps into both $\Tt C^*_{\II}(\Cc, \Sigma)$ and $C^*_{\II}(\Cc, \Sigma)$ the different notation ($\mathfrak{t}$ versus $\mathfrak{s}$) for the relevant families of partial isometries should prevent confusion.
\end{rmk}

\begin{rmk}\label{rmk:homotopies normalised}
	Suppose that $\Cc$ is a finitely aligned left-cancellative small category and let $\Sigma$ be a homotopy of $2$-cocycles for $\Cc$. Since we insist that all $2$-cocycles are normalised (\labelcref{rmk:coycles normalised}), we have $\Sigma_{\bullet}(r(c), c) = 1_\II = \Sigma_{\bullet}(c, s(c))$ for all $c$. Hence \labelcref{itm:cfc-mult} implies that if $(S, \CImap)$ is a $\Sigma$-twisted representation of $\Cc$ then $S_{r(c)} S_c = S_c = S_c S_{s(c)}$ for all $c \in \Cc$.
\end{rmk}

\begin{prop}[{cf.~\cite[Proposition~6.7]{SpiPath}}]\label[prop]{lem:spanning_set}
	Let $(S,\CImap)$ be a $\Sigma$-twisted representation of a finitely aligned left-cancellative small category $\Cc$ in a $C^*$-algebra $A$. Let $\Pp \coloneqq \{ \prod_{x \in F} S_xS_x^* \mid F \subseteq \Cc \text{ finite}\}$. Then
	\[
	C^*(S,\CImap) = \clsp\{ \CImap_{r(c)} (f)  S_{c}   S_{d}^* P \mid s(c) = s(d), f \in C(\II), P \in \Pp\}.
	\]
\end{prop}

\begin{proof} We follow the approach of~\cite[Proposition~6.7]{SpiPath}. Because the projections $S_x S_x^*$ for $x \in \Cc$ pairwise commute, $\Pp$ consists of projections and is closed under multiplication. Let $L \coloneqq \lsp \{ \CImap_{r(c)} (f)  S_{c}  S_{d}^* P \mid s(c) = s(d), f \in C(\II), P \in \Pp\}$. Then $L$ is a linear subspace of $C^*(S,\CImap)$, which by \labelcref{itm:cfc-one} and \labelcref{itm:cfc-mult} of \cref{dfn:twisted_lscs_algebra_homotopy} contains the generators of $C^*(S,\CImap)$ and their adjoints.
	
	It suffices to show that $L$ is $*$-closed and closed under multiplication. Let $\Mm$ denote the set $\{S_{c_1}S_{d_1}^* \cdots S_{c_k}S_{d_k}^* \mid c_i,d_i \in \Cc, k \ge 1\}$ of monomials in the $S_c$ and their adjoints.
	We first show that $\Mm \subseteq L$. We will say that a monomial $S_{c_1}S_{d_1}^* \cdots S_{c_k}S_{d_k}^*$ has length $k$.
	
	Monomials of length one lie in $L$ since each $ s_{r(c)}= \CImap_{r(c)}(1_{\II})$ and each $S_{r(d)} = \prod_{x \in \{r(d)\}} S_x S_x^*$.
	Fix a nonzero monomial $S_a S_b^* S_c S_d^*$ of length two.
	We claim that there exist $u_i,v_i \in \Cc$, and $P_i \in \Pp$ such that
	\[
	S_b^*S_c = \sum_{i}  S_{u_i} S_{v_i}^* P_i .
	\]
	Let $F$ be a finite independent set generating $b\sCc \cap c\sCc$, and write $F = \{x_1,\ldots, x_k\}$. Then
	\begin{equation}\label{eq:b_c_sandwich}
		S_b^* S_c = S_b^* S_b S_b^* S_c S_c^* S_c = S_b^* \Big(\bigvee_{i=1}^k S_{x_i} S_{x_i}^*\Big) S_c.
	\end{equation}
	
	Given commuting projections $Q_1,\ldots,Q_k$ we have
	\[
	\bigvee_{i=1}^k Q_i = Q_1 + Q_2(1-Q_1) + Q_3(1 - Q_1 \vee Q_2) + \cdots + Q_k\Big(1 - \bigvee_{j=1}^{k-1}Q_j \Big).
	\]
	Applying this with $Q_i = S_{x_i}S_{x_i}^*$ gives
	\begin{align}\label{eq:inclusion_exclusion}
		\begin{split}
			\bigvee_{i=1}^k S_{x_i} S_{x_i}^*
			&= \sum_{i = 1}^k S_{x_i} S_{x_i}^* \Big(1 - \bigvee_{j=1}^{i-1} S_{x_j}S_{x_j}^*  \Big)\\
			&= \sum_{i = 1}^k S_{x_i} S_{x_i}^* \Big(1 - \sum_{j=1}^{i-1} (-1)^{j-1} \sum_{1 \le r_1 < \cdots < r_{j} <  i}  S_{x_{r_1}} S_{x_{r_1}}^* \cdots S_{x_{r_j}} S_{x_{r_j}}^* \Big).
		\end{split}
	\end{align}
	Since each $x_i \in F \subseteq b\,\Cc \cap c\,\Cc$, for each $i$ there exist $b_i,c_i \in \Cc$ such that $x_i = b b_i = c c_i$. So,
	\begin{align*}
		S_b^* S_{x_i} S_{x_i}^* S_c &= S_b^* S_{b} S_{b_i} S_{c_i}^* S_c^* S_c = S_{b_i} S_{c_i}^* \in L \quad \text{and}\\
		S_c^* S_{x_i} S_{x_i}^* S_c &= S_c^* S_{c} S_{c_i} S_{c_i}^* S_c^* S_c = S_{c_i} S_{c_i}^* \in L.
	\end{align*}
	Since each $x_{r_j} \in c \sCc$ we have $S_{x_{r_j}}S_{x_{r_j}}^* \le S_c S_c^*$, so
	\begin{align*}
		S_b^*S_{x_{r_1}} S_{x_{r_1}}^* \cdots S_{x_{r_j}} S_{x_{r_j}}^* S_c
		&= S_b^*S_{x_{r_1}} S_{x_{r_1}}^* S_c \Big(\prod_{l=2}^j S_c^* S_{x_{r_l}}S_{x_{r_l}}^* S_c\Big)\\
		&=  S_{b_{r_1}} S_{c_{r_1}}^* \Big(\prod_{l=2}^j S_{c_{r_l}}S_{c_{r_l}}^*\Big).
	\end{align*}
	Hence, \labelcref{eq:b_c_sandwich} and \labelcref{eq:inclusion_exclusion} give the claim.
	
	By the claim, there exist $u_i,v_i \in \Cc$, and $P_i \in \Pp$ such that
	\[
	S_aS_b^*S_c S_d^*= \sum_{i} S_a  S_{u_i} S_{v_i}^* P_i S_d^*.
	\]
	If $r(x) = s(d)$ then $S_d S_x S_x^* = S_d S_x S_x^* S_d^* S_d = S_{dx} S_{dx}^* S_{d}$. In particular, for each $i$ there exists $P_i' \in \Pp$ such that $P_i S_d^* = S_d^* P_i'$. Hence, by \labelcref{itm:cfc-mult},
	\[
	S_a  S_{u_i} S_{v_i}^* P_i  S_d^*
	= \CImap_{r(a)}(\Sigma_{\bullet}(a,u_i)\ol{\Sigma_{\bullet}(d,v_i)}) S_{a u_i} S_{d v_i}^* P_i' \in L
	\]
	for each $i$ and so $S_aS_b^*S_c S_d^* \in L$.
	
	Now, suppose inductively that every monomial of length $k$ belongs to $L$. By the inductive hypothesis there exist $a_i,b_i \in \Cc$, $f_i \in C(\II)$ and $P_i \in \Pp$ such that
	\begin{align*}
		\textstyle
		S_{c_1}S_{d_1}^* \cdots S_{c_{k+1}}S_{d_{k+1}}^*
		&= S_{c_1}S_{d_1}^* \sum_{i} \CImap_{r(a_i)}(f_i)  S_{a_i}  S_{b_i}^* P_i\\
		&= \sum_{i} \CImap_{r(c_1)}(f_i) S_{c_1}S_{d_1}^* S_{a_i} S_{b_i}^* P_i.
	\end{align*}
	Since monomials of length two belong to $L$ we have now established that $\Mm \subseteq L$.
	
	The set $\Mm$ is closed under multiplication and $\Pp \subseteq \Mm$ so $L$ contains all elements of the form $\CImap_{r(c_1)}(f)S_{c_1}S_{d_1}^* \cdots S_{c_k}S_{d_k}^*P$. To see that $L$ is closed under multiplication fix spanning elements $\CImap_{r(c_1)}(f_1)S_{c_1} S_{d_1}^* P_1$ and $\CImap_{r(c_2)}(f_2)S_{c_2} S_{d_2}^* P_2$. Applying \labelcref{itm:cfc-commute} twice gives
	\[
	\CImap_{r(c_1)}(f_1)S_{c_1} S_{d_1}^* P_1 \CImap_{r(c_2)}(f_2)S_{c_2} S_{d_2}^* P_2
	=
	\CImap_{r(c_1)}(f_1f_2)S_{c_1} S_{d_1}^* P_1 S_{c_2} S_{d_2}^* P_2 \in L.
	\]
	So $L$ is closed under multiplication. To see that it is $*$-closed note that
	repeated application of \labelcref{itm:cfc-commute} gives
	\[
	(\CImap_{r(c_1)}(f_1)S_{c_1} S_{d_1}^* P_1)^* = \CImap_{r(c_1)}(f_1^*)P_1^* S_{d_1} S_{c_1}^* \in L. \qedhere
	\]
\end{proof}

\begin{cor}\label[cor]{cor:approx_id}
	Let $(S,\CImap)$ be a $\Sigma$-twisted representation of a finitely aligned left-cancellative small category $\Cc$ in a $C^*$-algebra $A$. The net \[\Big(\sum_{v \in F} S_v\Big)_{F \subseteq \Cc^0\text{ finite}} = \Big(\sum_{v \in F} \CImap_v(1_\II)\Big)_{F \subseteq \Cc^0\text{ finite}}\] is an approximate identity for $C^*(S,\CImap)$.
\end{cor}
\begin{proof}
	Fix $a \in C^*(S, \CImap)$ and $\varepsilon > 0$. By \cref{lem:spanning_set} there exist a finite $F \subseteq \Cc^0$ and $a_0 \in \lsp\{\CImap_{r(c)} (g) S_{c} S_{d}^* P \mid r(c) \in F, s(c) = s(d), g \in C(\II), P \in \Pp\}$ such that $\|a - a_0\| < \frac{\varepsilon}{2\|f\|}$. If $v \in \Cc^0 \setminus F$, then $S_v a_0 = 0$. Fix a finite $G \subseteq \Cc^0$ such that $F \subseteq G$. Since the $S_v$ are mutually orthogonal, $\sum_{v \in G} S_v$ is a projection, and since $a_0 \in \lsp\{\CImap_{r(c)} (g) S_{c} S_{d}^* P \mid r(c) \in F, s(c) = s(d), g \in C(\II), P \in \Pp\}$, we have $\sum_{v \in G} S_v a_0 = a_0$. Using these facts at the second and third inequalities, we calculate:
	\begin{align*}
		\Big\| \sum_{v \in G} S_v a - a \Big\|
		&\le \Big\| \sum_{v \in G} S_v (a - a_0) \Big\| + \Big\| \big(\sum_{v \in G} S_v a_0\big) - a \Big\| \\
		&\le \Big\| \sum_{v \in G} S_v \Big\| \| a - a_0 \| + \|a_0 - a\|
		\le 2 \| a - a_0\| = \varepsilon.
	\end{align*}
	Hence, $\sum_v S_v$ converges strictly to the identity in $\Mm(C^*(S,\CImap))$.
\end{proof}

\begin{cor}
	\label[cor]{lem:C(I)-algebra}
	Let $(S,\CImap)$ be a $\Sigma$-twisted representation of a finitely aligned left-cancellative small category $\Cc$ in a $C^*$-algebra $A$. Then there is a unique $*$-homomorphism $\hatCImap$ from $C(\II)$ to the centre $\Zz \Mm(C^*(S,\CImap))$ of the multiplier algebra of $C^*(S, \CImap)$ satisfying
	$
	\hatCImap(f) S_v = \CImap_v(f)
	$
	for all $v \in \Cc^0$ and $f \in C(\II)$. This homomorphism is unital and makes $C^*(S,\CImap)$ into a $C(\II)$-algebra.
\end{cor}
\begin{proof}
	Fix $f \in C(\II)$. We show that $\sum_v \CImap_v(f)$ converges strictly to a multiplier $\hatCImap(f)$. If $f = 0$ this is trivial, so suppose $f \not= 0$. Fix $a \in C^*(S,\CImap)$ and $\varepsilon > 0$. By \cref{cor:approx_id}, there is a finite $F \subseteq \Cc^0$ such that for all finite $G \subseteq \Cc^0$ containing $F$, we have $\big\|\sum_{v \in G\setminus F} S_v a\big\| < \frac{\varepsilon}{2\|f\|}$. Since $\CImap_v(f)S_w = 0$ if $v \not= w$, it follows that for finite $G \subseteq \Cc^0$ containing $F$, we have $\|\sum_{v \in G \setminus F} \CImap_v(f) a\| = \|\sum_{v \in G \setminus F} \CImap_v(f) \sum_{w \in G \setminus F} S_w a\|$. Since $\CImap_v(C(\II))\CImap_w(C(\II)) = \{0\}$ for $v \ne w$, for any finite $H \subseteq \Cc^{0}$ the universal property of direct sums gives a $*$-homomorphism $\oplus_H \CImap_v \colon \bigoplus_{v \in H} C(\II) \to C^*(S,\CImap)$ such that $(\oplus_H \CImap_v)((f_v)_{v \in H}) = \sum_{v \in H} \CImap_v(f_v)$. Hence, $\|\sum_{v \in H} \CImap_v(f_v)\| \le \|(f_v)_{v \in H}\| = \sup_{v \in H}\|f_v\|$.
	In particular, if $f \in C(\II)$ we have $\|\sum_{v \in H} \CImap_v(f)\| \le \|f\|$. Thus
	\begin{align*}
		\Big\|\sum_{v \in G \setminus F} \CImap_v(f) \sum_{w \in G \setminus F} S_w a\Big\| 
		&\le \Big\|\sum_{v \in G \setminus F} \CImap_v(f)\Big\|\,\Big\|\sum_{w \in G \setminus F} S_w a\Big\|\\
		&< \|f\| \frac{\varepsilon}{2\|f\|} = \frac{\varepsilon}{2}.
	\end{align*}
	Hence, the standard $\varepsilon/2$ argument shows that the net $(\sum_{v \in F} \CImap_v(f) a)_{F\subseteq \Cc^0\text{ finite}}$ is Cauchy.
	So $(\sum_{v \in F} \CImap_v(f))_{F\subseteq \Cc^0\text{ finite}}$ converges strictly in $\Mm(C^*(S,\CImap))$ to a multiplier $\hatCImap(f)$. Using \labelcref{itm:cfc-one} (of \cref{dfn:twisted_lscs_algebra_homotopy}) and mutual orthogonality of the $S_v$ we have
	\begin{align}\label{eq:tauv_times_vertex_projection}
		\begin{split}
			\hatCImap(f)S_v
			&= \Big(\lim_F \sum_{w \in F} \CImap_w(f) \Big) S_v
			= \Big(\lim_F \sum_{w \in F} \CImap_w(f) S_w \Big) S_v\\
			&= \lim_F \CImap_v(f)S_v
			= \CImap_v(f).
		\end{split}
	\end{align}
	Since each $\CImap_v$ is a $*$-homomorphism so is $\hatCImap \colon f \mapsto \hatCImap(f)$.
	
	To see that $\hatCImap(C(\II))$ is central, fix $f \in C(\II)$. Applying~\eqref{eq:tauv_times_vertex_projection} to $f^*$ and taking adjoints
	gives $S_v \hatCImap(f) = \CImap_v(f)$. So for $c \in \Cc$, relation~\labelcref{itm:cfc-commute} gives
	\[
	\hatCImap(f) S_c = \hatCImap(f) S_{r(c)} S_c = \CImap_{r(c)}(f) S_c = S_c \CImap_{s(c)}(f) = S_c S_{s(c)} \hatCImap(f) = S_c \hatCImap(f).
	\]
	Applying this to $f^*$ and taking adjoints shows that $\hatCImap(f)S^*_c = S^*_c \hatCImap(f)$ as well. For $v \in \Cc^0$ and $g \in C(\II)$, we
	have $\hatCImap(f) \CImap_v(g) = \hatCImap(f) S_v \CImap_v(g) = \CImap_v(fg) = \CImap_v(gf) = \CImap_v(g)S_v \hatCImap(f) = \CImap_v(g) \hatCImap(f)$.
	So $\hatCImap(f)$ commutes with all of the generators
	of $C^*(S,\CImap)$. Since $f \in C(\II)$ was arbitrary, it follows from strict continuity that $\hatCImap(C(\II))$ is central in $\Mm(C^*(S,\CImap))$.
	
	We have $\hatCImap(1_{\II}) = \sum_v \CImap_v(1_{\II}) = \sum_{v} S_v$, so \cref{cor:approx_id} says that $\hatCImap$ is unital.	
	Since a multiplier on a $C^*$-algebra is determined uniquely by its action on any approximate identity, and in particular on the approximate identity of \cref{cor:approx_id}, $\hatCImap$ is the unique $*$-homomorphism satisfying the desired relation.
\end{proof}

\begin{rmk}\label[rmk]{rmk:I-relation_relationships}
	In \cref{prop:relation_relationships} we established that our definition of a covariant $\sigma$-twisted representation of $\Lambda \bowtie \Gg$ matches the notion of a Cuntz--Krieger $(\Gg,\Lambda; \sigma)$ representation as in \cite{MS?}. It also allows us to relate covariant $\Sigma$-twisted representations of $\Lambda \bowtie \Gg$ with the relations \labelcref{itm:TCK3}~and~\labelcref{itm:CK}. Suppose that $(\Gg, \Lambda)$ is a self-similar action of a groupoid on a row-finite $k$-graph with no sources and that $\Sigma$ is a homotopy of $2$-cocycles on $\Lambda \bowtie \Gg$. Thinking of $\II$ as a category consisting solely of objects, and of $\II \times (\Lambda \bowtie \Gg)^2$ as the composable pairs in the product category $\II \times (\Lambda \bowtie \Gg)$, the map $\Sigma$ is a unitary-valued function into $C(\II)$, so $\omega \coloneqq \hatCImap \circ \Sigma$ given by \cref{lem:C(I)-algebra} is of the form needed to apply \cref{prop:relation_relationships}. Consequently every $\Sigma$-twisted representation $S$ of $\Lambda \bowtie \Gg$ satisfies~\labelcref{itm:TCK3} and every covariant $\Sigma$-twisted representation $S$ satisfies~\labelcref{itm:CK}. This also allows us to use the following alternative form of~\labelcref{itm:TCK3} for $\Sigma$-twisted representations: if $\mu,\nu \in \Lambda$ then
	\begin{align}
		S^*_\mu S_\nu
		&= S^*_\mu S_\mu S^*_\mu S_\nu S^*_\nu S_\nu \nonumber\\
		&= S^*_\mu \sum_{\mu\alpha = \nu\beta \in \MCE(\mu,\nu)}S_{\mu\alpha} S^*_{\nu\beta} S_\nu \nonumber\\
		&= \sum_{\mu\alpha = \nu\beta \in \MCE(\mu,\nu)} \hatCImap(\Sigma_{\bullet}(\mu,\alpha))^*\hatCImap(\Sigma_{\bullet}(\nu,\beta)) S^*_\mu S_\mu S_\alpha S^*_\beta S^*_\nu S_\nu \nonumber\\
		&= \sum_{\mu\alpha = \nu\beta \in \MCE(\mu,\nu)} \hatCImap(\Sigma_{\bullet}(\mu,\alpha)^*\Sigma_{\bullet}(\nu,\beta)) S_\alpha S^*_\beta.\label{eq:alternate_mult}
	\end{align}
\end{rmk}

\begin{cor}\label[cor]{cor:spanning_set_ssa}
	Fix $k \ge 0$. Suppose that $\Gg$ is a countable discrete groupoid acting self-similarly on a finitely aligned $k$-graph $\Lambda$, and let $\Sigma$ be a homotopy of $2$-cocycles on $\Lambda \bowtie \Gg$. Fix a $\Sigma$-twisted representation $(S,\CImap)$ of $\Lambda \bowtie \Gg$.
	Then
	\[
	C^*(S,\CImap) = \clsp \{\hatCImap(f) S_{\lambda} S_g S_{\mu}^* \mid \lambda,\mu  \in \Lambda, g \in s(\lambda) \Gg s(\mu), f\in C(\II)\}.
	\]
\end{cor}
\begin{proof}
	Let $\Pp$ be the set of \cref{lem:spanning_set} for $\Cc = \Lambda \bowtie \Gg$. We claim that $\lsp\Pp = \lsp\{S_{\lambda}S_{\lambda}^* \colon \lambda \in \Lambda\}$. We clearly have $\supseteq$. For the reverse, note that for each $c \in \Lambda \bowtie \Gg$ there exist unique $\lambda \in \Lambda$ and $g \in \Gg$ such that $S_{c} = \ol{\hatCImap(\Sigma_{\bullet}(\lambda,g))} S_{\lambda} S_g$. By \cref{lem:twisted_category_algebra_identities}, $S_{c} S_{c}^* = S_{\lambda} S_g S_g^* S_{\lambda}^* = S_{\lambda} S_{\lambda}^*$.
	If $\lambda,\mu \in \Lambda$, then~\labelcref{eq:range_prod_sum} gives
	\begin{align*}
		S_{\lambda}S_{\lambda}^* S_{\mu}S_{\mu}^*
		= \sum_{\nu \in \lambda \vee \mu} S_{\nu}S_{\nu}^*.
	\end{align*}
	Hence, $\lsp\{S_{\lambda}S_{\lambda}^* \colon \lambda \in \Lambda\}$ is closed under multiplication. Since it contains each $S_x S^*_x$ it follows that it contains $\lsp\Pp$.
	For $\mu,\nu \in \Lambda$, \cref{eq:range_prod_sum} and then~\labelcref{itm:cfc-mult} give
	\begin{align*}
		S_{\mu}^*S_{\nu}S_{\nu}^*
		&= S_{\mu}^* S_{\mu} S_{\mu}^* S_{\nu}S_{\nu}^* 
		= S_{\mu}^* \sum_{\mu \alpha \in \mu \vee \nu} S_{\mu \alpha} S_{\mu \alpha}^*\\
		&= \sum_{\mu \alpha \in \mu \vee \nu} S_{\mu}^*  \hatCImap(\Sigma(\mu,\alpha))^* S_{\mu} S_{\alpha} S_{\alpha}^* S_{\mu}^*\hatCImap(\Sigma_{\bullet}(\mu,\alpha)).
	\end{align*}
	Since the range of $\hatCImap$ is central and $\Sigma$ is unitary valued, and since $S_{\mu}^* S_\mu = S_{s(\mu)} = S_{r(\alpha)}$ by~\labelcref{itm:cf-source} and $S_{r(\alpha)} S_\alpha = S_\alpha$ by \cref{rmk:homotopies normalised}, we obtain
	\begin{equation}\label{eq:spanning_calc}
		S_{\mu}^*S_{\nu}S_{\nu}^*
		= S_{r(\alpha)} \Big( \sum_{\mu \alpha \in \mu \vee \nu} S_{\alpha}S_{\alpha}^*  \Big) S_{\mu}^*
		= \Big( \sum_{\mu \alpha \in \mu \vee \nu} S_{\alpha}S_{\alpha}^*  \Big) S_{\mu}^*.
	\end{equation}
	Combining \cref{lem:spanning_set} with the claim above gives
	\begin{align*}
		C^*(S,\CImap) = \clsp\{\hatCImap(f) S_{c} S_{d}^* S_{\nu} S_{\nu^*} \mid c,d &\in \Lambda \bowtie \Gg,\\
		& \nu \in \Lambda, g \in s(\lambda) \Gg s(\mu), f\in C(\II)\}.
	\end{align*}
	Given a spanning element $\hatCImap(f) S_{c} S_{d}^* S_{\nu} S_{\nu^*}$ as above, we can factor $c = \lambda g$ and $d = \mu h$ with $\lambda,\mu \in \Lambda$ and $g,h \in \Gg$ and then use \cref{lem:twisted_category_algebra_identities}\labelcref{itm:identities_1} and~\labelcref{itm:cfc-mult} followed by centrality of the range of $\hatCImap$ to write
	\begin{align*}
		\hatCImap(f) S_{c} S_{d}^* S_{\nu} S_{\nu^*}
		&= \hatCImap(f) \hatCImap(\Sigma_{\bullet}(\lambda,g))^* S_{\lambda} S_g \hatCImap(\Sigma_{\bullet}(h,h^{-1})) \hatCImap(\Sigma_{\bullet}(\mu,h))\\
		&\qquad \qquad S_{h^{-1}} S_{\mu}^*   S_{\nu} S_{\nu^*}\\
		&= \hatCImap(f \Sigma_{\bullet}(\lambda,g)^* \Sigma_{\bullet}(h,h^{-1})\Sigma_{\bullet}(g,h^{-1}) \Sigma_{\bullet}(\mu,h))\\
		&\qquad \qquad 
		S_{\lambda} S_{gh^{-1}} S_{\mu}^*  S_{\nu} S_{\nu^*},
	\end{align*}
	and we obtain
	\[
	C^*(S,\CImap) = \clsp\{\hatCImap(f) S_{\lambda} S_g S_{\mu}^* S_{\nu} S_{\nu^*} \mid \lambda,\mu,\nu \in \Lambda, g \in s(\lambda) \Gg s(\mu), f\in C(\II)\}.
	\]
	Using~ \labelcref{eq:spanning_calc} and using centrality of the image of $\hatCImap$ as we did in the preceding calculation, we can write a spanning element $\hatCImap(f) S_{\lambda} S_g S_{\mu}^* S_{\nu} S_{\nu^*}$ as
	\begin{align*}
		\hatCImap(f) &S_{\lambda} S_g S_{\mu}^* S_{\nu} S_{\nu^*}\\
		&= \hatCImap(f) S_{\lambda} S_g \Big( \sum_{\mu \alpha \in \mu \vee \nu} S_{\alpha}S_{\alpha}^*  \Big) S_{\mu}^*\\
		&= \sum_{\mu \alpha \in \mu \vee \nu}\Big( \hatCImap\big(f\Sigma_{\bullet}(\lambda, g \rhd\alpha)\Sigma_{\bullet}(g,\alpha)\Sigma_{\bullet}(g\rhd\alpha,g\lhd\alpha)^*\Sigma_{\bullet}(\mu,\alpha)\big) \\
		&\qquad \qquad S_{\lambda(g\rhd\alpha)} S_{g\lhd\alpha} S_{\mu\alpha}^* \Big),\\
		&\in \clsp \{\hatCImap(f) S_{\lambda} S_g S_{\mu}^* \mid \lambda,\mu  \in \Lambda, g \in s(\lambda) \Gg s(\mu), f\in C(\II)\}.\qedhere
	\end{align*}
\end{proof}

\begin{lem}[cf. {\cite[Theorem~3.3]{Gil15}}]\label[lem]{lem:epsilon_evaluation}
	Let $\Sigma$ be a homotopy of $2$-cocycles on a finitely aligned left-cancellative small category $\Cc$. Let $(\mathfrak{s}, \frakCImap)$ denote a universal $\Sigma$-twisted representation of $\Cc$ in $C^*_\II(\Cc, \Sigma)$, and let $s \colon \Cc \to C^*(\Cc, \Sigma_t)$ be a universal covariant $\Sigma_t$-twisted representation of $\Cc$ (see Definitions \ref{dfn:twisted_lcsc_algebra}~and~\ref{dfn:univ_twisted_lcsc_alg}). For $v \in \Cc^0$ define $\CImap_v \colon C(\II) \to C^*(\Cc, \Sigma_t)$ by $\CImap_v(f) = f(t)s_v$. For each $t \in \II$ there is a surjective homomorphism
	$\varepsilon_t \colon  C_{\II}^*(\Cc;\Sigma) \to C^*(\Cc, \Sigma_t)$ satisfying
	\[
	\varepsilon_t(\mathfrak{s}_{c}) = s_c \quad \text{ and } \quad \varepsilon_t(\frakCImap_v(f)) = f(t)s_v
	\]
	for all $c \in \Cc$, $f \in C(\II)$, and $v \in \Cc^0$.
	We have $\ker(\varepsilon_t) = C_{\II}^*(\Cc;\Sigma) \frakhatCImap(C_0(\II\setminus\{t\}))$, and $\varepsilon_t$ factors through an isomorphism of the fibre $C^*_{\II}(\Cc, \Sigma)_t$ of $C_{\II}^*(\Cc;\Sigma)$---regarded as a $C(\II)$-algebra as in \cref{lem:C(I)-algebra}---onto $C^*(\Cc, \Sigma_t)$.
\end{lem}
\begin{proof}
	We claim that $(s,\CImap)$ is a covariant $\Sigma$-twisted representation of $\Cc$. The map $s$ satisfies~\labelcref{itm:cf-source}, \labelcref{itm:cf-range}~and~\labelcref{itm:cf-ck} of \cref{dfn:twisted_lcsc_algebra} because it is a covariant representation of $\Cc$.  We have $\CImap_v(1_\II) = 1 s_v = s_v$, which is \labelcref{itm:cfc-one} (of \cref{dfn:twisted_lscs_algebra_homotopy}). For composable $c_1, c_2 \in \Cc$, we have $s_{c_1} s_{c_2} = \Sigma_t(c_1, c_2) s_{c_1c_2} = \big(\Sigma_{\bullet}(c_1, c_2)\big)(t) s_{c_1c_2} = \CImap_{r(c_1)}\big(\Sigma_{\bullet}(c_1, c_2)\big) s_{c_1c_2}$, giving~\labelcref{itm:cfc-mult}. So $(s,\CImap)$ is a covariant $\Sigma$-twisted representation of $\Cc$ as claimed, and hence the universal property of $C_{\II}^*(\Cc;\Sigma)$ induces a $*$-homomorphism $\varepsilon_t \colon  C_{\II}^*(\Cc;\Sigma) \to C^*(\Cc, \Sigma_t)$ satisfying $\varepsilon_t(\mathfrak{s}_{c}) = s_c$ for all $c \in \Cc$ and $\varepsilon_t(\frakCImap_v(f)) = \CImap_v(f)s_v = f(t)s_v$ for all $v \in \Cc^0$ and $f \in C(\II)$. The image of $\varepsilon_t$ contains the generators of $C^*(\Cc, \Sigma_t)$ so $\varepsilon_t$ is surjective. 	
	
	For the final statement let $A \coloneqq C_{\II}^*(\Cc;\Sigma)$ and $I_t = \pi(C_0(\II \setminus \{t\}))A$. The fibre of $A$ at $t \in \II$ is $A_t \coloneqq A/I_t$.
	By definition of $\varepsilon_t$ and $\frakCImap$, we have $I_t \subseteq \ker(\varepsilon_t)$. Hence, $\varepsilon_t$ descends to a $*$-homomorphism $\ol{\varepsilon_t} \colon A_t \to C^*(\Cc, \Sigma_t)$.
	
	We use the universal property of $C^*(\Cc, \Sigma_t)$ to construct an inverse to $\ol{\varepsilon_t}$. Define $S \colon \Cc \to A_t$ by $S_c \coloneqq \mathfrak{s}_c + I_t$ for $c \in \Cc$. We claim that $S$ is a covariant $\Sigma_t$-twisted representation of $\Cc$. The $S_c$ are partial isometries because the $\mathfrak{s}_c$ are. For $(c_1,c_2) \in \Cc^2$,
	\[
	{\frakhatCImap}(\Sigma_{\bullet}(c_1,c_2))\mathfrak{s}_{c_1c_2} - \Sigma_t(c_1,c_2)\mathfrak{s}_{c_1c_2} = {\frakhatCImap}\big(\Sigma_{\bullet}(c_1,c_2) - \Sigma_t(c_1,c_2) 1_{\II}\big) \mathfrak{s}_{c_1c_2} \in I_t,
	\]
	so
	\[
	S_{c_1}S_{c_2} = {\frakhatCImap}(\Sigma_{\bullet}(c_1,c_2))\mathfrak{s}_{c_1c_2} + I_t = \Sigma_t(c_1,c_2)\mathfrak{s}_{c_1c_2} +I_t = \Sigma_t(c_1,c_2) S_{c_1c_2}.
	\]
	So $S$ satisfies~\labelcref{itm:cf-mult}. The $\mathfrak{s}_c$ satisfy~\labelcref{itm:cf-source}, \labelcref{itm:cf-range}~and~\labelcref{itm:cf-ck} because $(\mathfrak{s}, \frakCImap)$ is a covariant $\Sigma$-twisted representation. Consequently their images $S_c$ under the homomorphism $a \mapsto a + I_t$ satisfy the same relations. So $S$ is a $\Sigma_t$-twisted representation of $\Cc$ in $A_t$ as claimed, and the universal property of $C^*(\Cc, \Sigma_t)$ gives a $*$-homomorphism $\psi \colon C^*(\Cc, \Sigma_t) \to A_t$ such that $\psi(s_c) = S_c$ for all $c \in \Cc$. Since each $\CImap_v(f) - f(t)\CImap_v(1_\II) \in I_t$, each $\CImap_v(f) + I_t = f(t) \CImap_v(1_\II) = f(t) s_v$. Since $A$ is generated by $\{\mathfrak{s}_c : c \in \Cc\} \cup \{\frakCImap_v(f) : v \in \Cc^0\text{ and } f \in C(\II)\}$, the quotient $A_t$ is generated by the images of these elements, and we deduce that $A_t = C^*(\{\mathfrak{s}_c + I_t : c \in \Cc\})$. Hence the maps $\psi$ and $\ol{\varepsilon_t}$ are mutually inverse on generators and hence are mutually inverse isomorphisms.
\end{proof}

\subsection{Invariance of \texorpdfstring{\boldmath$K$}{K}-theory under \texorpdfstring{\boldmath$2$}{2}-cocycle homotopy}

\begin{dfn}
	Let $\Cc$ be a finitely aligned left-cancellative small category. We say that \emph{$K_*(C^*(\Cc, \bullet))$ is constant along homotopies} if for any homotopy $\Sigma$ of 2-cocycles and any $t \in \II$, the homomorphisms
	\[
	K_*(\varepsilon_t) \colon K_*(C^*_{\II}(\Cc, \Sigma)) \to K_*(C^*(\Cc, \Sigma_t))
	\]
	in $K$-theory induced by the map $\varepsilon_t$ of \cref{lem:epsilon_evaluation}
	are isomorphisms.
\end{dfn}

\begin{rmk}
	The property that $K_*(C^*(\Cc, \bullet))$ is constant along homotopies is a property of $\Cc$, but we could not think of good suggestive terminology that better emphasises this. Note that while the condition is stated in terms of $K_*(\varepsilon_t)$ for all $t$, in fact $K_*(C^*(\Cc, \bullet))$ is constant along homotopies if and only if  $K_*(\varepsilon_0)$ is isomorphism for every homotopy $\Sigma$ of $2$-cocycles.
\end{rmk}

Suppose that $(\Cc,\Dd)$ is a matched pair of categories.
If $\Sigma \colon \II \times (\Dd \bowtie \Cc)^2 \to \TT$ is a homotopy of $2$-cocycles on $\Dd \bowtie \Cc$, then the restriction $\Sigma|_{\Cc} \colon \II \times \Cc^2 \to \TT$ defined by
$
(\Sigma|_{\Cc})_t (c_1,c_2) \coloneqq \Sigma_t (c_1,c_2)
$
for $(c_1,c_2) \in \Cc^2 \subseteq (\Dd \bowtie \Cc)^2$ is a homotopy of $2$-cocycles on $\Cc$. We frequently just write $\Sigma$ for this restriction. We can now state our main theorem.

\begin{thm}\label[thm]{thm:cocycle_homotopy_independence} For any jointly faithful self-similar action of a countable discrete amenable groupoid $\Gg$ on a row-finite $k$-graph $\Lambda$ with no sources, $K_*(C^*(\Lambda \bowtie \Gg, \bullet))$ is constant along homotopies.
\end{thm}

Our strategy for proving \cref{thm:cocycle_homotopy_independence} is to apply Elliott's inductive Five-Lemma argument \cite{Ell84} to the Pimsner exact sequence in $K$-theory \cite[Theorem~4.8]{Pim97}. The following proposition provides the base case of the induction.

\begin{prop}\label[prop]{prop:groupoids_K_cts}
	Let $\Gg$ be a countable discrete amenable groupoid. Then $K_*(C^*(\Gg, \bullet))$ is constant along homotopies.
\end{prop}
\begin{proof}
	Fix a homotopy $\Sigma$ of $2$-cocycles on $\Gg$ and a point $t \in \II$. Choose a subset $X$ of $\Gg^{0}$ that meets each $\Gg$-orbit exactly once: that is $|X \cap r(\Gg x)| = 1$ for all $x \in \Gg^0$.
	
	By the argument of \cref{lem:twisted_category_algebra_identities}, for all $(g,h^{-1}) \in \Gg^2$ we have
	\[
	\mathfrak{s}_g \mathfrak{s}_{h}^* = \frakhatCImap(\Sigma_{\bullet}(g,h^{-1})\ol{\Sigma_{\bullet}(h,h^{-1})}) \mathfrak{s}_{gh^{-1}}.
	\]
	Hence, \cref{lem:spanning_set} implies that $C_{\II}^*(\Gg,\Sigma)$ is spanned by elements of the form $\frakhatCImap(f)\mathfrak{s}_g$ where $f \in C(\II)$ and $g \in \Gg$. Fix such an element and a finite subset $F \subseteq X$. Then
	$
	\sum_{x \in F} \mathfrak{s}_x \frakhatCImap(f) \mathfrak{s}_g  = 0
	$
	if $r(g) \notin X$, and if $r(g) \in X$ then $\sum_{x \in F} \mathfrak{s}_x \frakhatCImap(f) \mathfrak{s}_g = \frakhatCImap(f)\mathfrak{s}_g$ whenever $\{r(g)\} \subseteq F$. So
	\[
	\sum_{x \in X} \mathfrak{s}_x \frakhatCImap(f) \mathfrak{s}_g
	=
	\begin{cases}
		\frakhatCImap(f) \mathfrak{s}_g & \text{if } r(g) \in X\\
		0 & \text{otherwise}.
	\end{cases}
	\]
	Taking adjoints shows that
	\[
	\frakhatCImap(f) \mathfrak{s}_g 	\sum_{x \in X} \mathfrak{s}_x
	=
	\begin{cases}
		\frakhatCImap(f) \mathfrak{s}_g & \text{if } s(g) \in X\\
		0 & \text{otherwise}.
	\end{cases}
	\]
	An approximation argument similar to the one in \cref{lem:C(I)-algebra} shows that $P_X \coloneqq \sum_{x \in X} \mathfrak{s}_x$ is a multiplier projection of $C_{\II}^*(\Gg,\Sigma)$. Fix $y \in \Gg^0$.
	Since $X$ meets each orbit there exists $g \in \Gg$ such that $r(g) = y$ and $s(g) \in X$. So $\mathfrak{s}_y = \mathfrak{s}_{g}
	P_X \mathfrak{s}_{g}^*$. Hence, $P_X$ is full.
	
	Similarly, the characteristic function $\bone_X$ is a full multiplier projection of $C^*(\Gg, \Sigma_t)$. Calculating on generators as above shows that the extension of $\varepsilon_t$ to multiplier algebras carries $P_X$ to $\bone_X$.
	Because $K$-theory is Morita invariant, it suffices to show that $\varepsilon_t \colon P_X C^*_{\II}(\Gg,\Sigma)P_X \to \bone_X C^*(\Gg,\Sigma_t) \bone_X$ induces an isomorphism in $K$-theory.
	
	Since $X$ intersects each $\Gg$-orbit once, for distinct $x,y \in X$ and $g \in \Gg$ if $r(g) = x$, then $s(g) \ne y$. So $\mathfrak{s}_x \mathfrak{s}_g \mathfrak{s}_y = 0$.
	Consequently,
	\[
	P_X C^*_{\II}(\Gg,\Sigma) P_X = \bigoplus_{x \in X} \mathfrak{s}_x C^*_{\II}(\Gg,\Sigma) \mathfrak{s}_x \cong \bigoplus_{x \in X} C^*_{\II}(x \Gg x, \Sigma).
	\]
	
	Since $\Gg$ is amenable, each $x \Gg x$ is amenable (see \cref{rmk:amenability}), and hence each $C^*_{\II}(x \Gg x, \Sigma)$ is canonically isomorphic to the $C^*$-algebra $C^*_r(G,\Omega)$ described in the paragraph following \cite[Theorem~0.3]{ELPW10} for $G = x \Gg x$ and $\Omega = \Sigma|_{x \Gg x}$.
	By \cite[Corollary~9.2]{HigsonKasparov} each $x \Gg x$ satisfies the Baum--Connes conjecture with coefficients. So \cite[Theorem~0.3]{ELPW10} implies that the restriction of $\varepsilon_t$ to $C^*_{\II}(x \Gg x, \Sigma)$ induces an isomorphism $K_*(\varepsilon_t)$ from $K_*(C^*_{\II}(x \Gg x, \Sigma))$ to $K_*(C^*(x \Gg x, \Sigma_t))$. Since $K$-theory respects direct sums, it follows that
	\[
	K_*(\varepsilon_t) \colon K_*\Big(\bigoplus_{x \in X} C^*_{\II}(x \Gg x, \Sigma)\Big) \to  K_*\Big( \bigoplus_{x \in X}C^*(x \Gg x, \Sigma_t)\Big)
	\]
	is an isomorphism.
\end{proof}

\begin{rmk}
	\cref{prop:groupoids_K_cts} does not imply that the $K$-theory of $C^*(\Gg, \sigma)$ is independent of $\sigma$. There exist discrete amenable groups $G$ and (non-homotopic) $2$-cocycles $\sigma_1,\sigma_2$ on $G$ such that $K_*(C^*(G,\sigma_1)) \not\cong K_*(C^*(G,\sigma_2))$, \cite[Proposition~3.11]{PR92}.
\end{rmk}

For the inductive step in the proof of \cref{thm:cocycle_homotopy_independence}, given a row-finite $(k+1)$-graph $\Lambda$ with no sources, we need a Cuntz--Pimsner model, in the sense of the universal property of Cuntz--Pimsner algebras described in \cite[Theorem~3.12]{Pim97}, for $C_{\II}^*(\Lambda \bowtie \Gg, \Sigma)$, whose coefficient algebra is $C_{\II}^*(\Gamma \bowtie \Gg, \Sigma)$ for a rank-$k$ subgraph $\Gamma$ of $\Lambda$. We establish such a model in \cref{prop:CP-model} after proving two preliminary lemmas; the proof of \cref{thm:cocycle_homotopy_independence} then appears on page~\pageref{proof:main_theorem}. From here until \cref{prop:CP-model}, we work in the following setup. 

\noindent\textbf{Standing assumptions:} Fix a jointly faithful self-similar action of a discrete amenable groupoid $\Gg$ on a row-finite $(k+1)$-graph $\Lambda$ with no sources ($k$ could be 0). Let $\Gamma = \Lambda^{\NN^k}$ and $E^* = \Lambda^{\NN e_{k+1}}$, so that $\Lambda = E^* \bowtie \Gamma$ as in \cref{ex:k-graph_zs}. Let $\Sigma$ be a homotopy of $2$-cocycles on $\Lambda \bowtie \Gg$.

\begin{lem}
	\label[lem]{lem:homotopy_twist_inclusion}
	The inclusion $\Gamma \bowtie \Gg \hookrightarrow \Lambda \bowtie \Gg$ induces an injective $*$-homomorphism $\Pi \colon C^*_{\II}(\Gamma \bowtie \Gg, \Sigma) \to C^*_{\II}(\Lambda \bowtie \Gg, \Sigma)$.
\end{lem}
\begin{proof}
	Let $(\mathfrak{s}^\Gamma,\frakCImap^\Gamma)$ and $(\mathfrak{s}^\Lambda,\frakCImap^\Lambda)$ denote universal $\Sigma$-twisted representations of $\Gamma \bowtie \Gg$ and $\Lambda \bowtie \Gg$, respectively. Then $(\mathfrak{s}^{\Lambda}|_{\Gamma \bowtie \Gg},\frakCImap^{\Lambda})$ satisfies \labelcref{itm:cfc-one}, \labelcref{itm:cfc-commute}, \labelcref{itm:cfc-mult}, and \labelcref{itm:cf-source} (see Definitions \ref{dfn:twisted_lcsc_algebra}~and~\ref{dfn:twisted_lscs_algebra_homotopy}) as these relations for $\Lambda \bowtie \Gg$ are identical to those for $\Gamma \bowtie \Gg$ on elements of the latter. \Cref{lem:self_similar_concordance} implies that $\Gamma \bowtie \Gg \hookrightarrow \Lambda \bowtie \Gg$ is concordant, so the argument of the second paragraph of the proof of \cref{prop:toeplitz_map} shows that $(\mathfrak{s}^{\Lambda}|_{\Gamma \bowtie \Gg},\frakCImap^{\Lambda})$ satisfies~\labelcref{itm:cf-range}. To see that it satisfies~\labelcref{itm:cf-ck}, fix $v \in \Lambda^0$ and a finite set $F \subseteq v\Gamma$ that is exhaustive in $\Gamma$. Then \cref{lem:self-similar_exhaustive} shows that $F$ is exhaustive in $\Lambda$, and since $\mathfrak{s}^{\Lambda}$ satisfies~\labelcref{itm:cf-ck} with respect to $\Lambda$, it follows that $\mathfrak{s}^{\Lambda}_v = \bigvee_{c \in F} \mathfrak{s}^\Lambda_v (\mathfrak{s}^\Lambda_v)^*$, as required. Hence, the inclusion $\Gamma \bowtie \Gg \hookrightarrow \Lambda \bowtie \Gg$ induces a  $*$-homomorphism $\Pi \colon C^*_{\II}(\Gamma \bowtie \Gg, \Sigma) \to C^*_{\II}(\Lambda \bowtie \Gg, \Sigma)$.
	
	To see that $\Pi$ is injective, note that $\Pi \circ \frakhatCImap^{\Gamma} = \frakhatCImap^{\Lambda} \circ \Pi$, so $\Pi$ is a homomorphism of $C(\II)$-algebras. So it induces homomorphisms $\Pi_t \colon C^*(\Gamma \bowtie \Gg, \Sigma_t) \to C^*(\Lambda \bowtie \Gg, \Sigma_t)$. Since the norm on a $C(\II)$-algebra is the supremum norm \cite[Proposition~C.23 and Theorem~C.26]{Wil07}, it suffices to show that each $\Pi_t$ is injective, which is \cref{cor:subalg}.
\end{proof}

Let $A \coloneqq \Pi(C_{\II}^*(\Gamma \bowtie \Gg, \Sigma)) \subseteq C^*_{\II}(\Lambda \bowtie \Gg, \Sigma)$, the range of the homomorphism $\Pi$ of \cref{lem:homotopy_twist_inclusion}.

\begin{lem}\label[lem]{lem:correspondence}
	The subspace
	\[
	X \coloneqq \clsp\{\mathfrak{s}_e a \colon e \in E^1,\, a \in A\}
	\]
	of $C_{\II}^*(\Lambda \bowtie \Gg, \Sigma)$
	is a right Hilbert $A$-module with right action and inner product satisfying
	\[
	\mathfrak{s}_e a \cdot b = \mathfrak{s}_e ab \quad \text{ and } \quad \langle \mathfrak{s}_e a \mid \mathfrak{s}_f b \rangle = \delta_{e,f} a^*\mathfrak{s}_{s(e)} b.
	\]
	The map $\varphi \colon A \to \Ll_A(X)$ defined by  $\varphi(a) \mathfrak{s}_e b = a \mathfrak{s}_e b$ makes $(\varphi,X)$ a nondegenerate $C^*$-correspondence over $A$. The map $\varphi$ is injective and takes values in the generalised compact operators $\Kk_A(X)$.
\end{lem}
\begin{proof}
	Routine calculations show that $X$ is a right pre-Hilbert $A$-module under the given right action and inner product. For each linear combination $\sum_i \mathfrak{s}_{e_i} a_i \in X$ we have, using \labelcref{itm:cf-range} (of \cref{dfn:twisted_lcsc_algebra}) at the third equality,
	\begin{align*}
		\Big\|\sum_i \mathfrak{s}_{e_i} a_i \Big\|_{X}^2
		&= \Big\| \sum_{i,j} \langle   \mathfrak{s}_{e_i} a_i \mid \mathfrak{s}_{e_j} a_j \rangle \Big\|_{C^*(\Gamma \bowtie \Gg, \Sigma)}
		= \Big\| \sum_{i,j} \delta_{e_i,e_j} a_i^* \mathfrak{s}_{s(e_i)} a_j \Big\|_{C^*(\Gamma \bowtie \Gg, \Sigma)}\\
		&=\Big\|\Big(\sum_{i} \mathfrak{s}_{e_i}a_i\Big)^*\Big(\sum_j \mathfrak{s}_{e_j} a_j\Big)\Big\|_{C^*(\Gamma \bowtie \Gg, \Sigma)}
		=
		\Big\|\sum_{i } \mathfrak{s}_{e_i} a_i \Big\|_{C_{\II}^*(\Lambda \bowtie \Gg, \Sigma)}^2,
	\end{align*}
	so the norm on $X$ agrees with the $C^*$-norm on $C^*_{\II}(\Lambda \bowtie \Gg, \Sigma)$. In particular, $X$ is complete.
	
	We claim that $A X \subseteq X$. It suffices to show that $a\mathfrak{s}_e \in X$ for each $a$ in the spanning set described in \cref{cor:spanning_set_ssa} and each $e \in E^1$.
	For this, fix $e \in E^1$, $f \in C(\II)$, and $(\gamma_1,g,\gamma_2) \in \Gamma \fibre{r}{s} \Gg \fibre{s}{s} \Gamma$. Then \labelcref{eq:alternate_mult} gives
	\begin{align*}
		\frakhatCImap(f) \mathfrak{s}_{\gamma_1} \mathfrak{s}_g \mathfrak{s}_{\gamma_2}^* \mathfrak{s}_e
		&= \sum_{\gamma_2 \alpha = e \beta \in \gamma_2 \vee e} \frakhatCImap\big(f\, \ol{\Sigma_{\bullet}(\gamma_2,\alpha)}\Sigma_{\bullet}(\beta,e)\big) \mathfrak{s}_{\gamma_1} \mathfrak{s}_{g}  \mathfrak{s}_{\alpha} \mathfrak{s}_{\beta}^*.
	\end{align*}
	For each term in this sum,
	\begin{align*}
		\mathfrak{s}_{\gamma_1} \mathfrak{s}_{g}  \mathfrak{s}_{\alpha} \mathfrak{s}_{\beta}^* & =
		\frakhatCImap\big(\Sigma_{\bullet}(g,\alpha) \ol{\Sigma_{\bullet}(g \la \alpha,g \ra \alpha)}\big) \mathfrak{s}_{\gamma_1} \mathfrak{s}_{g\la \alpha} \mathfrak{s}_{g \ra \alpha} \mathfrak{s}_{\beta}^*\\
		&=  \frakhatCImap\big(\Sigma_{\bullet}(g,\alpha) \ol{\Sigma_{\bullet}(g \la \alpha,g \ra \alpha)} \Sigma_{\bullet}(\gamma_1, g \la \alpha)\big)
		\mathfrak{s}_{\gamma_1 (g \la \alpha)} \mathfrak{s}_{g \ra \alpha} s_{\beta}^*.
	\end{align*}
	Since $d(\gamma_2) \in \NN^{k} \times \{0\}$ and $d(e) \in \{0\} \times \NN$, $d(\gamma_2) \vee d(e) = d(\gamma_2) + d(e)$, so $d(\alpha) = d(e)$. That is, $\alpha \in E^1$. Since self-similar actions are degree preserving we have $g \la \alpha \in E^1$. By uniqueness of factorisations in $\Lambda$ there exist unique $\alpha' \in E^1$ and $\gamma_1' \in \Gamma$ such that $\gamma_1(g \la \alpha) = \alpha' \gamma_1'$. So
	\begin{align*}
		\mathfrak{s}_{\gamma_1 (g \la \alpha)} \mathfrak{s}_{g \ra \alpha} \mathfrak{s}_{\beta}^*
		&= \mathfrak{s}_{\alpha'\gamma_1'} \mathfrak{s}_{g \ra \alpha} \mathfrak{s}_{\beta}^*
		= \mathfrak{s}_{\alpha'} \Big( \frakhatCImap\big(\ol{\Sigma_{\bullet}(\alpha',\gamma_1')}\big) \mathfrak{s}_{\gamma_1'}  \mathfrak{s}_{g \ra \alpha} \mathfrak{s}_{\beta}^*\Big) \in \mathfrak{s}_{\alpha'} A \subseteq  X.
	\end{align*}
	Hence, $\frakhatCImap(f) \mathfrak{s}_{\gamma_1} \mathfrak{s}_g \mathfrak{s}_{\gamma_2}^* \mathfrak{s}_e \in X$.
	
	To see that $X$ is a $C^*$-correspondence we must show that $\varphi(a)$ is adjointable for each $a \in A$. It suffices to show that $\varphi(a^*)$  is an adjoint for $\varphi(a)$. Given spanning elements $\mathfrak{s}_e b$ and $\mathfrak{s}_f c$ of $X$ we have
	\[
	\langle \mathfrak{s}_e b \mid \mathfrak{s}_f c \rangle = \delta_{e,f} b^* \mathfrak{s}_{s(e)} c = (\mathfrak{s}_e b)^*(\mathfrak{s}_f c)
	\]
	so linearity and continuity imply that $\langle \xi \mid \eta \rangle = \xi^* \eta$, computed in $C^*_{\II}(\Lambda \bowtie \Gg, \Sigma)$, for all $\xi,\eta \in X$. In particular, $\langle \varphi(a) \xi \mid \eta \rangle = (a \xi^*) \eta = \xi^* a^* \eta = \langle \xi \mid \varphi(a^*) \eta \rangle$, so $\varphi(a^*)$ is indeed an adjoint for $\varphi(a)$.

	The $(k+1)$-graph $\Lambda$ is row-finite with no sources, so for each $v \in \Gamma^0$ the set $ vE^1 \subseteq v \Lambda$ is finite and exhaustive \cite[Proof of Lemma~B.2]{RSY04}.
	Since $\mathfrak{s}_e = \mathfrak{s}_e \mathfrak{s}_{s(e)} \in X$, we have $\mathfrak{s}_e \mathfrak{s}_e^{*} \in \Kk_A(X)$ for each $e \in E^1$. Since $\Lambda$ has no sources, and since $\mathfrak{s}$ satisfies~\labelcref{itm:CK} by \cref{rmk:I-relation_relationships}, for each $v \in \Gamma^0$ we have,
	\[
	\mathfrak{s}_v = \sum_{e \in v E^1} \mathfrak{s}_e \mathfrak{s}_e^* \in \Kk_A(X).
	\]
	Since, by \cref{cor:approx_id},  $(\sum_{v \in F} \mathfrak{s}_v)_{F \subseteq E^0 \text{ finite}}$ is an approximate unit for $C_{\II}^*(\Lambda \bowtie \Gg, \Sigma)$, we obtain $\varphi(A) \subseteq \Kk_A(X)$. This approximate unit belongs to $A$, so the left action is nondegenerate. If $\varphi(a) = 0$ then
	\begin{align*}
		a \mathfrak{s}_v = \sum_{e \in vE^1} a \mathfrak{s}_e \mathfrak{s}_e^*
		=  \sum_{e \in vE^1} \varphi(a) \mathfrak{s}_e \mathfrak{s}_e^* = 0,
	\end{align*}
	so $a = \lim_{F} \sum_{v \in F} a\mathfrak{s}_v  = 0$. That is, $\varphi$ is injective.
\end{proof}

\begin{ntn}
	By construction, $A$ and $X$ are subsets of $C^*_{\II}(\Lambda \bowtie \Gg, \Sigma)$. We show that there is an isomorphism of $C^*_{\II}(\Lambda \bowtie \Gg, \Sigma)$ onto the Cuntz--Pimsner algebra $\Oo_X$ that carries $A$ to its canonical image $j_A(A)$ in $\Oo_X$ and carries $X$ to its canonical image $j_X(X)$. To do so, we invoke the universal property of $\Oo_X$, with respect to the inclusion maps $a \mapsto a : A \to C^*_{\II}(\Lambda \bowtie \Gg, \Sigma)$ and $x \mapsto x : X \to C^*_{\II}(\Lambda \bowtie \Gg, \Sigma)$. We formalise this by giving these maps names: $\iota_A : A \to C^*_{\II}(\Lambda \bowtie \Gg, \Sigma)$ is the inclusion $A \owns a \mapsto \iota_A(a) \coloneqq a \in C^*_{\II}(\Lambda \bowtie \Gg, \Sigma)$, and similarly $\iota_X : X \to C^*_{\II}(\Lambda \bowtie \Gg, \Sigma)$ is the inclusion $X \owns x \mapsto \iota_X(x) \coloneqq x \in C^*_{\II}(\Lambda \bowtie \Gg, \Sigma)$.
\end{ntn}

\begin{prop}\label[prop]{prop:CP-model} Let $A \coloneqq \Pi(C_{\II}^*(\Gamma \bowtie \Gg, \Sigma)) \subseteq C_{\II}^*(\Lambda \bowtie \Gg, \Sigma)$ and let $X \subseteq C_{\II}^*(\Lambda \bowtie \Gg, \Sigma)$ be as in \cref{lem:correspondence}.
	The inclusion maps $\iota_A \colon A \to C_{\II}^*(\Lambda \bowtie \Gg, \Sigma)$ and $\iota_X \colon X \to C_{\II}^*(\Lambda \bowtie \Gg, \Sigma)$ comprise a covariant representation of $X$ in $C_{\II}^*(\Lambda \bowtie \Gg, \Sigma)$. Let $(j_A,j_X)$ be a universal covariant representation of $X$ in $\Oo_{X}$. There is a unique isomorphism $\Xi \colon \Oo_{X} \to  C_{\II}^*(\Lambda \bowtie \Gg, \Sigma)$ satisfying $\Xi(j_A(a)) = \iota_A(a) = a$ and $\Xi(j_X(x)) = \iota_X(x) = x$ for all $a \in A$ and $x \in X$.
\end{prop}

\begin{proof}
	By construction $(\iota_A,\iota_X)$ is a representation of the correspondence $(\varphi,X)$ in $C^*_{\II}(\Lambda \bowtie \Gg, \Sigma)$. Since $\iota_A(\mathfrak{s}_v) = \sum_{e \in vE^1} \mathfrak{s}_e \mathfrak{s}_e^* = \sum_{e \in vE^1} \iota_X(\mathfrak{s}_e)\iota_X(\mathfrak{s}_e)^*$ for all $v \in \Lambda^0$, and since, by \cref{cor:approx_id}, $(\sum_{v \in F} \mathfrak{s}_v)_{F \subseteq \Lambda^0\text{ finite}}$ is an approximate identity for $A$, the pair $(\iota_A,\iota_X)$ is covariant \cite[Lemma~A.3.16]{Mun20}. Hence, there is a unique $*$-homomorphism $\Xi \colon \Oo_{X} \to C_{\II}^*(\Lambda \bowtie \Gg, \Sigma)$ such that $\Xi(j_A(a)) = \iota_A(a)$ and $\Xi(j_X(x)) = \iota_X(x)$ for $a \in A$ and $x \in X$.
	
	The standard universal-property argument (see, for example, \cite[Proposition~2.1]{Rae05}) yields a strongly continuous action $\beta \colon \TT\to \Aut(C^*_{\II}(\Lambda \bowtie \Gg, \Sigma))$ such that $\beta_z( \frakhatCImap(f) \mathfrak{s}_{\lambda g}) = z^{d(\lambda)_{k+1}} \frakhatCImap(f) \mathfrak{s}_{\lambda g}$ for all $f \in C(\II)$ and $\lambda g \in \Lambda \bowtie \Gg$.
	Since $\iota_A$ is injective and $\Xi$ intertwines $\beta$ with the gauge action on $\Oo_X$, \cite[Theorem~6.4]{Kat04} implies that $\Xi$ is injective. The subalgebra $\Xi(\Oo_{X})$ contains the generators of $C_{\II}^*(\Lambda \bowtie \Gg, \Sigma)$, so $\Xi$ is a $*$-isomorphism.
\end{proof}

\begin{proof}[Proof of \cref{thm:cocycle_homotopy_independence}] \label{proof:main_theorem}
	We induct on $k$. If $k = 0$ the result is \cref{prop:groupoids_K_cts}.
	
	Fix a
	jointly faithful self-similar action of a discrete amenable groupoid $\Gg$ on a row-finite $(k+1)$-graph $\Lambda$ with no sources and fix a homotopy $\Sigma \colon \II \times \Lambda \to \TT$ of $2$-cocycles.
	Let $\Lambda = E^* \bowtie \Gamma$ as in \cref{ex:k-graph_zs}. Suppose as an inductive hypothesis that $K_*(C^*(\Gamma \bowtie \Gg, \bullet))$ is constant along homotopies.
	
	Let $B \coloneqq C_{\II}^*(\Lambda \bowtie \Gg, \Sigma)$. As in \cref{lem:correspondence}, let $A \coloneqq \Pi(C_{\II}^*(\Gamma \bowtie \Gg, \Sigma)) \subseteq B$, and let $X \subseteq B$ be the $C^*$-correspondence constructed in that lemma, and write $\iota_A \colon A \to B$ and $\iota_X\colon X \to B$ for the inclusion maps. Let $(j_A,j_X)$ be a universal representation of $X$ in the Cuntz--Pimsner algebra $\Oo_X$.
	By \cref{prop:CP-model} there is an isomorphism $\Xi \colon \Oo_{X} \to B$ satisfying $\Xi(j_A(a)) = \iota_A(a) = a$ and $\Xi(j_X(x)) = \iota_X(x) = x$ for all $a \in A$ and $x \in X$.
	
	Fix $t \in \II$. Let
	\begin{equation}\label{eq:ItB}
		I_t^B \coloneqq B\frakhatCImap(C_0(\II \setminus \{t\})) = \frakhatCImap(C_0(\II \setminus \{t\}))B.
	\end{equation}
	Let $I_t^A \coloneqq A \cap I_t^B$. We show that $I_t^A = A\frakhatCImap(C_0(\II \setminus \{t\}))$. We have $A\frakhatCImap(C_0(\II \setminus \{t\})) \subseteq I_t^A$ by definition. To see that the reverse containment holds fix an approximate identity $(e_i)$ for $C_0(\II \setminus \{t\})$ and $a \in A \cap I_t^B$. Since $\frakhatCImap(e_i)$ is an approximate identity for $I_t^B$ we have $a = \lim_i a \frakhatCImap(e_i)$. So the Cohen--Hewitt Factorisation Theorem \cite[Theorem~2.5]{Hew64} (see also \cite{Coh59}) for the $\frakhatCImap(C_0(\II \setminus \{t\}))$-module $A \cap I^B_t$ implies that $a \in A\frakhatCImap(C_0(\II \setminus \{t\}))$. Consequently, we have $I_t^A = A \cap I_t^B = A\frakhatCImap(C_0(\II \setminus \{t\}))$.	
	Similarly, $X I_t^A \coloneqq \clsp\{x \cdot a \colon x  \in X, a \in I_t^A\}$ satisfies
	\begin{equation}\label{eq:fibre_factorisation}
		X I_t^A = X \cap I_t^B = X \frakhatCImap(C_0(\II \setminus \{t\})).
	\end{equation}
	
	We define $B_t \coloneqq B/ I_t^B$
	and write $q^B_t \colon B \to B_t$ for the quotient map. Define $A_t \coloneqq q^B_t(A)$. By the preceding paragraph there is an isomorphism $A/I_t^A \cong A_t$ that carries $a + I_t^A$ to $\iota_A(a) + I_t^B$. 	
	By \cite[Corollary~1.4]{KatIdeal} the subspace $X I_t^A$ is an $A$-submodule of $X$ and $ X / X I_t^A$ is a right $A_t$-module with respect to the obvious right $A_t$-action. Let $X_t \coloneqq q^B_t (X) \subseteq B_t$. By the preceding paragraph there is an isomorphism $ X / X I_t^A \cong X_t$ of right $A_t$-modules that carries $x + X I_t^A$ to $\iota_X(x) + I_t^B$.
	
	Since $\frakhatCImap(C_0(\II \setminus \{t\}))$ is central in $\Mm (B)$ we have $a\frakhatCImap(f)\cdot \mathfrak{s}_e b = a\mathfrak{s}_e b\frakhatCImap(f)$ for all $f \in C_0(\II \setminus \{t\})$, $a \in A$ and $\mathfrak{s}_eb \in X$.
	Hence by \labelcref{eq:fibre_factorisation} we have $a\frakhatCImap(f)\cdot \mathfrak{s}_e b \in X I_t^A$. In particular, the left action of $A$ on $X$ descends to an adjointable left action of $A_t$ on $X_{t}$, so $X_t$ is a $C^*$-correspondence over $A_t$. Since, for $v \in \Lambda^0$, we have $\mathfrak{s}_v = \sum_{e \in vE^1} \mathfrak{s}_e \mathfrak{s}_e^*$ in $B$, the same is true in $B_t$. Since, by \cref{cor:approx_id}, $(\sum_{v \in F} \mathfrak{s}_v + I_t^B)_{F \subseteq E^0 \text{ finite}}$ is an approximate unit in $B_t$, the left action of $A_t$ on $X_t$ is by compact operators.
	
	For injectivity of the left action suppose that $a + I^A_t \in A_t$ satisfies $(a + I^A_t) (\mathfrak{s}_e b + XI^A_t) = 0$ for all $e \in E^1$ and $b \in A$. That is,  $a\mathfrak{s}_e b \in XI^A_t $ for all $e \in E^1$ and $b \in A$. In particular, $a\mathfrak{s}_e \in XI^A_t \subseteq I_t^B$ for all $e \in E^1$. This implies that $a \mathfrak{s}_e \mathfrak{s}_e^* \in I_t^B$ for all $e \in E^1$, and so $a \mathfrak{s}_v = \sum_{e \in vE^1} a \mathfrak{s}_e \mathfrak{s}_e^* \in I_t^B$ for all $v \in E^0$. Since, by \cref{cor:approx_id}, $(\sum_{v \in F} \mathfrak{s}_v)_{F \subseteq E^0 \text{ finite}}$ is an approximate unit we have $a \in I_t^B$. Fix an approximate unit $(h_{i})$ for $C_0(\II \setminus \{t\})$. Then $(\frakhatCImap(h_i))$ is an approximate unit for $I_t^B$ by \eqref{eq:ItB}. So, $a = \lim \frakhatCImap(h_i) a \in I^A_t$. That is, $a + I_t^A = 0$.
	
	Let $(j_{A_t},j_{X_t})$ be a universal covariant representation of $X_t$ in $\Oo_{X_t}$. 	
	The inclusions $\iota_{A_t} \colon A_t \to B_t$ and $\iota_{X_t} \colon X_t \to B_t$ define a covariant representation $(\iota_{A_t},\iota_{X_t})$ of $X_t$ in $B_t$. By the universal property of $\Oo_{X_t}$ there is a unique $*$-homomorphism $\Xi_t \colon \Oo_{X_t} \to B_t$ such that
	\begin{equation}\label{eq:thm_1}
		\Xi_t(j_{A_t}(a + I_t^A)) = q_t^B(a)  \quad \text{ and } \quad \Xi_t(j_{X_t}(x + XI^A_t)) = q_t^B(x)
	\end{equation}
	for all $a \in A$ and $x \in X$. 	
	
	We construct an inverse to $\Xi_t$. The pair $(q_t^B \circ \iota_A, q_t^B \circ \iota_X)$ is a coisometric correspondence morphism (in the sense of \cite[Definition~1.3]{Bre10}) from $X$ to $X_t$, so by \cite[Proposition~1.4]{Bre10} there is a unique $*$-homomorphism $\Phi \colon \Oo_{X} \to \Oo_{X_t}$ satisfying $\Phi(j_{A}(a)) = j_{A_t}(a + I_t^A)$ and $\Phi(j_X(x)) = j_{X_t} (x + X I_t^A)$ for all $a \in A$ and $x \in X$. For each $a \in A$ and $x \in X$,
	\begin{equation}\label{eq:thm_2}
		\Phi \circ \Xi^{-1} (a) = j_{A_t}(a+I_t^A)\quad\text{ and }\quad
		\Phi \circ \Xi^{-1}(x) = j_{X_t}(x + X I_t^A).
	\end{equation}
	
	Since $\Oo_X$ is generated by $j_A(A)$ and $j_X(X)$, \cref{prop:CP-model} shows that $B$ is generated by $\Xi(j_A(A)) = A$ and $\Xi(j_X(X)) = X$. So $I^B_t$ is generated by $A\frakhatCImap(C_0(\II \setminus \{t\})) = I^A_t$, and we saw at~\labelcref{eq:fibre_factorisation} that $X\frakhatCImap(C_0(\II \setminus \{t\})) = X I^A_t$. If $a \in I^A_t$, then $\Phi \circ \Xi^{-1} (a) = j_{A_t}(a+I_t^A) = 0$, and if $x \in X I^A_t$, then $\Phi \circ \Xi^{-1}(x) = j_{X_t}(x +  X I^A_t) = 0$. Hence, $I^B_t \subseteq \ker(\Phi \circ \Xi^{-1})$, and so $\Phi \circ \Xi^{-1}$ descends to a $*$-homomorphism from $B_t$ to $\Oo_{X_t}$ such that $q_t^B(a) \mapsto j_{A_t}(a + I_t^A)$ and $q_t^B(x) \mapsto j_{X_t}(x + X I_t^A)$ for all $a \in A$ and $x \in X$. This map and $\Xi_t$ are mutually inverse on generators, so $\Xi_t$ is an isomorphism.
	
	By
	\cite[Theorem~8.6]{Kat04} and the final assertion of \cite[Theorem~4.4]{KPS18} (regarding naturality with respect to correspondence morphisms) there are homomorphisms $\partial$ and $\partial_t$ such that the diagram
	
	\begin{equation}\label{eq:large_diagram}
		\begin{tikzcd}[ampersand replacement=\&,column sep=25pt]
			{K_0(A)} \&\& {K_0(A)} \&\& {K_0(\Oo_X)} \\
			\& {K_0(A_t)} \& {K_0(A_t)} \& {K_0(\Oo_{X_t})} \\
			\& {K_1(\Oo_{X_t})} \& {K_1(A_t)} \& {K_1(A_t)} \\
			{K_1(\Oo_X)} \&\& {K_1(A)} \&\& {K_1(A)}
			\arrow["{\id - [X]}", from=1-1, to=1-3]
			\arrow["{K_0(q_t^B\circ \iota_A)}"{inner sep=0pt}, from=1-1, to=2-2]
			\arrow["{K_0( j_A)}", from=1-3, to=1-5]
			\arrow["{K_0(q_t^B\circ \iota_A)}", from=1-3, to=2-3, pos=0.4]
			\arrow["{K_0(\Phi)}"'{inner sep=0pt}, from=1-5, to=2-4]
			\arrow["\partial",from=1-5, to=4-5]
			\arrow["{\id - [X_t]}", from=2-2, to=2-3]
			\arrow["{K_0(j_{A_t})}", from=2-3, to=2-4]
			\arrow["\partial_t",from=2-4, to=3-4]
			\arrow["\partial_t",from=3-2, to=2-2]
			\arrow["{K_1( j_{A_t})}", from=3-3, to=3-2]
			\arrow["{\id - [X_t]}", from=3-4, to=3-3]
			\arrow["\partial",from=4-1, to=1-1]
			\arrow["{K_1(\Phi)}"'{inner sep=0pt}, from=4-1, to=3-2]
			\arrow["{K_1(q_t^B\circ \iota_A)}", from=4-3, to=3-3, pos=0.4]
			\arrow["K_1( j_A)", from=4-3, to=4-1]
			\arrow["{K_1(q_t^B\circ \iota_A)}"{inner sep=0pt}, from=4-5, to=3-4]
			\arrow["{\id - [X]}", from=4-5, to=4-3]
		\end{tikzcd}
	\end{equation}
	commutes and the rectangular six-term sequences in it are exact.

	Let $\varepsilon_t^A \colon C_{\II}^*(\Gamma \bowtie \Gg, \Sigma) \to C^*(\Gamma \bowtie \Gg, \Sigma_t)$ be as in \cref{lem:epsilon_evaluation} and recall that $\ker(\varepsilon_t^A) = C_{\II}^*(\Gamma \bowtie \Gg, \Sigma) \frakhatCImap(C(\II \setminus \{t\}))$. The isomorphism $\Xi$ takes $\ker(\varepsilon_t^A)$ to $\ker(q_t^B \circ \iota_A) = I_t^A$, so $\Xi$ descends to an isomorphism $\widetilde{\Xi} \colon C^*(\Gamma \bowtie \Gg, \Sigma_t) \to A_t$ such that $\widetilde{\Xi} \circ \varepsilon_t^A = (q_t^B \circ \iota_A) \circ \Xi$.
	By the inductive hypothesis, the maps $K_*(\varepsilon_t^A)$ are isomorphisms. Since $\widetilde{\Xi}$ and $\Xi$ are isomorphisms, so are the maps they induce in $K$-theory, so the $K_*(q_t^B \circ \iota_A) = K_*(\widetilde{\Xi}) \circ K_*(\varepsilon_t^A) \circ K_*(\Xi^{-1})$ are also isomorphisms. The Five Lemma applied to \labelcref{eq:large_diagram} implies that the $K_*(\Phi)$ are isomorphisms. Hence, the maps $K_*(\Xi_t \circ \Phi \circ \Xi^{-1})$ are isomorphisms. By equations \labelcref{eq:thm_1}~and~\labelcref{eq:thm_2} the map $\Xi_t \circ \Phi \circ \Xi^{-1}$ agrees with $q_t^B$ on generators, so the maps $K_*(q_t^B) \colon K_*(B) \to K_*(B_t)$ are isomorphisms.
	
	Let $\varepsilon_t^B \colon B \to C^*(\Lambda \bowtie \Gg, \Sigma_t)$ be as in \cref{lem:epsilon_evaluation}. Since $\ker(\varepsilon_t^B) = I_t^B = \ker(q_t^B)$ there is an isomorphism $B_t \cong C^*(\Lambda \bowtie \Gg, \Sigma_t)$ taking $q_t^B(b)$ to $\varepsilon_t^B(b)$ for all $b \in B$. It follows that $K_*(\varepsilon_t^B) \colon K_*(B) \to K_*(C^*(\Lambda \bowtie \Gg, \Sigma_t))$ is an isomorphism, completing the induction.	
\end{proof}

We recover a result of Gillaspy as a special case of \cref{thm:cocycle_homotopy_independence}.
\begin{cor}[{\cite[Theorem~4.1]{Gil15}}]
	Let $\Lambda$ be a row-finite $k$-graph with no sources. Then $K_*(C^*(\Lambda, \bullet))$ is constant along homotopies.
\end{cor}
\begin{proof}
	Let $\Gg = \Lambda^0$, the groupoid consisting of units. Then $\Gg$ is amenable, and $(\Gg,\Lambda)$ is a matched pair with $\Lambda \bowtie \Gg \cong \Lambda$. The result now follows from \cref{thm:cocycle_homotopy_independence}.
\end{proof}

\vspace{10pt}

\scriptsize{
	\noindent \textsc{Alexander Mundey}\\
	School of Mathematics and Applied Statistics,\\ University of Wollongong, NSW 2522, Australia. \\
	Email address: \texttt{amundey@uow.edu.au}\\
}

\scriptsize{
	\noindent \textsc{Aidan Sims}\\
	School of Mathematics and Applied Statistics,\\ University of Wollongong, NSW 2522, Australia. \\
	Email address: \texttt{asims@uow.edu.au}\\
}

\end{document}